\documentclass[11pt,reqno]{amsart}
\usepackage{amsthm}
\usepackage{amstext}
\usepackage{amssymb}
\usepackage{amsmath}

\usepackage{psfrag}

\usepackage{graphicx}
\usepackage{url}

\def\bl{\begin{lemma}}
\def\el{\end{lemma}}
\def\bth{\begin{theorem}}
\def\eth{\end{theorem}}
\def\bc{\begin{corollary}}
\def\ec{\end{corollary}}
\def\bcj{\begin{conjecture}}
\def\ecj{\end{conjecture}}
\def\bpr{\begin{proposition}}
\def\epr{\end{proposition}}
\def\bde{\begin{definition}}
\def\ede{\end{definition}}
\def\E{\mathbf{E}}

\def\Var{\mbox{\rm Var}}

\newcommand{\be}{\begin{eqnarray}}
\newcommand{\ee}{\end{eqnarray}}

\newcommand{\eqn}[2]{\begin{equation}\label{#1}{#2}\end{equation}}
\newcommand{\eqnst}[1]{\begin{equation*}{#1}\end{equation*}}
\newcommand{\eqnspl}[2]{\begin{equation}\label{#1}\begin{split}{#2}\end{split}\end{equation}}
\newcommand{\eqnsplst}[1]{\begin{equation*}\begin{split}{#1}\end{split}\end{equation*}}

\newcommand{\scrs}{\scriptstyle}

\newcommand{\eps}{{\mbox{$\epsilon$}}}
\newcommand{\R}{{\mathbb R}}

\newcommand{\Z}{{\mathbb Z}}

\newcommand{\T}{{\mathcal T}}
\newcommand{\TS}{{\mathcal{TS}}}
\newcommand{\F}{{\mathcal F}}
\newcommand{\U}{{\mathcal U}}
\newcommand{\V}{{\mathcal V}}
\newcommand{\I}{{\mathcal I}}
\newcommand{\G}{{\mathcal G}}
\newcommand{\Gtree}{\G_{\mathrm{tree}}}
\newcommand{\Ggood}{\G_{\mathrm{good}}}

\newcommand{\A}{{\mathcal A}}
\newcommand{\B}{{\mathcal B}}

\newcommand{\D}{{\mathcal D}}
\newcommand{\cH}{{\EE}}

\newcommand{\RR}{{\mathcal R}}

\newcommand{\XX}{{\mathcal X}}
\newcommand{\YY}{{\mathcal Y}}

\newcommand{\EE}{{\mathcal E}}
\newcommand{\NN}{{\mathcal N}}

\renewcommand{\and}{\hbox{ {\rm and} }}

\newcommand{\prob}{\mbox{\bf P}}
\newcommand{\p}{\mbox{\bf p}}
\newcommand{\q}{\mbox{\bf q}}

\newcommand{\LL}{\mathcal{L}}

\newcommand{\lr}{\leftrightarrow}

\newtheorem{theorem}{Theorem}
\newtheorem{definition}{Definition}[section]
\newtheorem{lemma}{Lemma}[section]

\newtheorem{corollary}[lemma]{Corollary}
\newtheorem{proposition}[theorem]{Proposition}
\newtheorem{conjecture}[theorem]{Conjecture}

\theoremstyle{definition}
\numberwithin{equation}{section}

\def\Reff{R_{\rm eff}}

\newcommand{\lra}{\leftrightarrow}

\newcommand{\lf}{\lfloor}
\newcommand{\rf}{\rfloor}
\newcommand{\cR}{\mathcal{R}}

\usepackage{color}

\begin{document}
\title[Electrical resistance of the critical branching random walk]{Electrical resistance of the low dimensional critical branching random walk}
\author{Antal A. J\'arai} \author{Asaf Nachmias}

\begin{abstract} We show that the electrical resistance between the origin and generation $n$ of the incipient infinite  oriented branching random walk in dimensions $d<6$ is $O(n^{1-\alpha})$ for some universal constant $\alpha > 0$. This answers a question of Barlow, J\'arai, Kumagai and Slade \cite{BJKS08}.

%

%
\end{abstract}

\maketitle

\section{{\bf Introduction}}

We study the electrical resistance of the trace of oriented critical branching random walk (BRW) in low dimensions. This trace is obtained by drawing a critical Galton-Watson tree $\T$ conditioned to survive forever and randomly mapping it into $\Z^d \times \Z^+$ in the following manner: we initialize by mapping the root of $\T$ to $(o,0)$ and recursively, if $V \in \T$ was mapped to $(x, n)$ and $U \in \T$ is a child of $V$, then we map $U$ to $(y,n+1)$ where $y$ is chosen according to a symmetric random walk distribution (we assume that this distribution has an exponential moment). Denote by $\Phi:\T \to \Z^d \times \Z^+$ this random mapping. The trace we consider in this paper is the graph induced by set of edges $\{\Phi(V), \Phi(U)\}$ for every edge $\{U,V\}$ of $\T$. 

It follows from the work of Barlow, J\'arai, Kumagai and Slade \cite[Example 1.8(iii)]{BJKS08} (who studied the much more difficult model of critical oriented percolation (OP)) that when $d>6$, the electrical resistance between the root and generation $n$ in the BRW is linear in probability. This enabled them to calculate various exponents describing the behavior of the simple random walk on the trace. In particular, they show that the mean hitting time at graph distance $n$ is $\Theta(n^3)$, that the spectral dimension equals $4/3$ and more, see \cite{BJKS08}.

They asked \cite[Section 1.4 (iii)]{BJKS08} whether the resistance of the critical BRW is still linear in $n$ in dimensions $4 < d \leq 6$, that is, in any dimension above the critical dimension $4$ of OP \cite{HHS1, HHS2, HHS3, HS1}. Here we answer their question by showing that the resistance is $O(n^{1-\alpha})$ when $d \leq 5$.

\begin{theorem}
\label{mainthmint}
Let $R(n)$ denote the expected effective resistance between the origin and generation $n$
of a branching random walk in dimension $d<6$ with progeny distribution that has mean $1$, positive variance
and finite third moment, conditioned to survive forever. There exists a universal constant $\alpha>0$ such that
$$ R(n) = O(n^{1-\alpha}) \, .$$
\end{theorem}

Unlike our firm understanding of anomalous diffusion in high dimensions \cite{BJKS08, Kes, KN08}, random fractals in
low dimensions are not (stochastically) finitely ramified. That is, we do not see pivotal edges at every scale. This makes their analysis more challenging, even in the case of the critical BRW which is one of the simplest models of statistical
physics. Our argument heavily relies on the built-in independence and self-similarity of
the model to obtain recursive inequalities for the resistance. We first show that intersections within the trace occur at every scale (see Figure \ref{fig:K-tree-good} and Theorem \ref{goodthingshappen}); these intersections exist only when $d<6$. Secondly, we show that the branches leading to each intersection are themselves distributed as BRW, allowing us to bound the electrical circuit using the parallel law and to form recursive estimates (Theorem \ref{goodblock}). There are additional technical difficulties to overcome. For instance, when intersections do not occur, the resistance is stochastically larger than it is unconditionally and one needs to get adequate bounds on it. Calculating the precise polynomial exponent which determines the growth of $R(n)$ when $d<6$ remains a challenging open problem.

As mentioned before, it is believed that OP in $d=5$ behaves similarly to BRW hence we expect an analogue of Theorem \ref{mainthmint} to hold. Presumably, the general setup (illustrated in Figure \ref{fig:K-tree-good}) and proving existence of intersections (Theorem \ref{goodthingshappen}) can be done for OP (based on results of \cite{HHS1, HHS2}). However, due to the lack of distributional self-similarity in OP it seems difficult to obtain recursive bounds (that is, an analogue of Theorem \ref{goodblock}). Furthermore, we do not know whether the exponent determining the growth of the resistance in OP in $d=5$ should be the same as the one for BRW (assuming they both exist).

It is easy to see (and stated in \cite{BJKS08}) that the volume up to generation $n$ of the BRW trace is of order $\Theta(n^2)$ in probability. Hence, Theorem \ref{mainthmint} together with the commute time identity \eqref{commutetime} shows that the mean exit time of the simple random walk on the BRW trace from the ball of radius $n$ in graph distance is at most $O(n^{3-\alpha})$, i.e., much faster than the $\Theta(n^3)$ in dimensions $d>6$, see \cite{BJKS08}. In fact, if one calculated the exponent determining the growth of the resistance, then many other random walk exponents (such as the spectral dimension, walk dimension etc.) could be determined, see \cite{KM}. In particular, if the resistance exponent exists, it follows from our results that the spectral dimension is strictly larger than $4/3$. \\

\noindent{\bf Remark 1.} We emphasize that the exponent $\alpha>0$ of Theorem \ref{mainthmint} is universal in the sense that it does not depend on the progeny or random walk distributions.


\noindent {\bf Remark 2.} By projecting the trace to $\Z^d$ we get a similar result for the usual (non-oriented) branching random walk: the effective resistance between the origin and the particles of generation $n$ is $O(n^{1-\alpha})$ when $d \leq 5$. This is because the projection only decreases the effective resistance. By a similar argument, projecting $\Z^5$ into $\Z^d$ with $d<5$, we learn that it suffices to prove Theorem \ref{mainthmint} for $d=5$.

\subsection{Incipient infinite branching process}
Let $\{p(k)\}_{k \geq 0}$ be a progeny distribution of a Galton-Watson branching process.
Our assumptions on $\{p(k)\}$ are the following.
\begin{enumerate}
\item Criticality: $\sum_{k} k p(k)=1$.
\item Finite variance: $\sum_{k} k(k-1) p(k) = \sigma^2 \in (0,\infty)$.
\item Bounded third moment: $\sum_k k^3 p(k)  \leq C_3 < \infty$
\end{enumerate}
It is classical that under condition (i) (and that $p(1) < 1$) the branching process dies out with probability $1$. To construct the
incipient infinite branching process (IIBP), we simply condition on survival up to level $n$,
and take the weak limit of the measures obtained as $n \to \infty$. However, it will
be convenient for us to use an equivalent construction of the IIBP (see \cite{Kes, LPP}).

Consider an infinite path $(V_0, V_1, \ldots)$ and attach to
each vertex $V_i$ a critical branching process with progeny distribution $\tilde{p}$ in the first generation and $p$ afterwards, where $\tilde{p}$ is the size biased law of $p$ minus $1$, that is,
$$ \tilde{p} (k) = (k+1) p(k+1) \, .$$

\subsection{Incipient infinite branching random walk}
\label{ssec:IIBRW}
Let $\p^1(x,y)$ denote the $1$-step transition probability
of a random walk on $\Z^d$. We assume the following:
\begin{enumerate}
\item Exponential moment: $\sum_{x \in \Z^d} e^{b |x|} \p^1(o,x) < \infty$
for some $b > 0$.
\item Non-degeneracy: $\{ x \in \Z^d : \p^1(o,x) > 0 \}$ generates $\Z^d$ as a group.
\item Symmetry: $\p^1(x,y) = \p^1(y,x)$.
\end{enumerate}
We remark that we did not try to obtain the optimal condition on $\p^1$. In fact, conditions (ii) and (iii) are not essential for our proof,
and (i) can plausibly be replaced with a weaker condition, however, we opted to make the calculations smoother.
Likewise, we have not tried to optimize the moment condition on $p(k)$.

Given a rooted tree $T$ we define a random mapping $\Phi : T \to \Z^d \times \Z_+$
which we will call henceforth a ``random walk'' mapping. Firstly, $\Phi$ maps the root of $T$
to $(o,0)$ and recursively, given a vertex $V$ of $T$ at height $h$ and
its mapping $\Phi(V)=(x,h)$ we map each upward neighbor $U$ of $V$,
independently, by drawing a random neighbor $y$ of $x$, according to $\p^1(x, \cdot)$, and putting
$\Phi(U) = (y,h+1)$. The {\em incipient infinite branching random walk} (IIBRW)
is obtained by taking $T$ to be the IIBP.

For any tree $T$ we consider $\Phi(T)$ as a graph on the vertex set $\Z^d \times \Z_+$ and we add the edge $\{\Phi(U), \Phi(W)\}$ for any tree edge $\{U,W\}$ (there may be parallel edges). The {\em trace} of the IIBRW is simply $\Phi(T)$ where $T$ is the IIBP.
%
%
%

\subsection{Electrical resistance} We provide a brief background on the electric effective resistance of a network, for further information see \cite{LP}. Let $G=(V,E)$ be a finite connected graph with two marked vertices $a$ and $z$ (we assume here that all edge weights are $1$). The {\em effective resistance} between $a$ and $z$, denoted $\Reff(a \lr z)$, is the minimum energy $\EE(\theta)$ over all unit flows from $a$ to $z$, where $\EE(\theta) = \sum_{e \in E} \theta(e)^2$. The connection between this quantity and the simple random walk $\{S_t\}_{t \geq 0}$ on $G$ is evident via the identity (see \cite{LP}),
$$ \Reff(a \lr z) = {1 \over \deg(a) \prob_a(\tau_z < \tau_a^+)} \, ,$$
where $\deg(a)$ is the vertex degree of $a$, $\tau_z$ is the first visit time to $z$, $\tau_a^+$ is the first positive visit time to $a$ and $\prob_a$ is the simple random walk probability measure conditioned on $X_0=a$. Another useful connection is the {\em commute time identity} asserting that
\be\label{commutetime} \E_a \tau_z + \E_z \tau_a = 2 |E| \Reff(a \lr z) \, .\ee
We will frequently use the easy fact that the resistance satisfies the triangle inequality, that is, for any three vertices $x,y,z$ we have
\be\label{triangle} \Reff(x \lr z) \leq \Reff(x \lr y) + \Reff(y \lr z) \, .\ee
Lastly, we will use the {\em parallel law} for effective resistance stating that if $G_1=(V, E_1)$ and $G_2=(V,E_2)$ are two connected graphs on the same vertex set and $a,z \in V$, then the effective resistance $R_{1 \cup 2}$ between $a$ and $z$ in $(V,E_1 \cup E_2)$ (where we allow multiple edges in this union) satisfies
\be\label{parallel} {1 \over R_{1 \cup 2}} \geq {1 \over R_1} + {1 \over R_2} \, ,\ee
where $R_1, R_2$ are the effective resistances between $a$ and $z$ in $G_1, G_2$, respectively.

\subsection{Finite approximations} We use the following finite approximations to the IIBRW in order to establish recursions. \\

\noindent {\bf Definition.} Suppose $n \ge 1$ and $m \ge 2n$. Let $\T_{n,m}$ denote the following random tree:
\begin{enumerate}
\item A path of length $n$ (the {\em backbone}): $(V_0, \ldots, V_n)$ with a marked root $\rho=V_0$.
\item For each $0 \leq i \leq n-1$ attach to $V_i$ a critical branching process with progeny distribution $\tilde{p}$ in the first generation and $p$ afterwards, conditioned to die out before generation $m-i$ (i.e.~none of the vertices of the attached trees reach distance $m$ from $\rho$).
\end{enumerate}

\noindent The following is an important quantity in the proof. For $x\in \Z^d$ define
$$ \gamma(n,x) = \sup_{m \geq 2n} \E_{\T_{n,m}} \big [ \Reff\big ((o,0) \lra \Phi(V_n)\big ) \, | \,\Phi(V_n) = (x,n) \big ] \, ,$$
where we consider the resistance in the graph $\Phi(\T_{n,m})$.

It will be convenient to introduce the following norm on $\R^d$
adapted to the ``typical size'' of the random walk displacements.
We do this in order to conveniently obtain a universal estimate
on $\alpha$ of Theorem \ref{mainthmint}, but the reader may just
assume that $\p^1$ is the transition matrix of the nearest-neighbor
simple random walk and that the norm below is the Euclidean norm.

Let $Q_{ij} = \sum_{x \in \Z^d} x_i x_j \p^1(o,x)$ be the covariance
matrix of the step distribution, and let $Q^{-1}$ denote the
inverse of $Q$. We define
\eqn{e:norm}
{ \| x \|
  := \sqrt{\frac{1}{d} \sum_{i,j=1}^d x_i Q^{-1}_{ij} x_j} \, .}

The main effort in this paper is the following theorem.

\begin{theorem}
\label{mainthm}
Assume $d \leq 5$. There exists a universal constant $\alpha \in(0,1/2)$ and also
$A = A (\sigma^2, C_3, \p^1) < \infty$ such that for all $n\geq 1$
$$ \gamma(n,x) \leq A n ^{1-\alpha} \Big ( { \| x \|^2 \over n } \vee 1 \Big )^\alpha \, .$$
\end{theorem}

\noindent {\em Remark.} Note that we cannot expect $\gamma(n,x)$ to be $O(n^{1-\alpha})$ for all $x$. Indeed, when $\|x\| = \Theta(n)$, then conditioned on $\Phi(V_n)=(x,n)$ the projection of the path $\Phi(V_0), \ldots, \Phi(V_n)$ onto $\Z^d$ has positive speed and since this conditioning does not affect the mapping of the trees hanging on $V_i$, there will be little intersections and we expect the resistance then to be linear in $n$. Theorem \ref{mainthm} will be proved by induction, hence it has to contain an estimate valid for all $x\in \Z^d$. \\

\noindent{{\bf Proof of Theorem \ref{mainthmint} assuming Theorem \ref{mainthm}.}} Recall that in the construction of the IIBRW we attach to the backbone unconditional critical trees, whereas in the definition of $\T_{n,m}$ we attach critical trees conditioned not to reach a certain level. However, when $n$ is fixed and $m \to \infty$ the distribution of these critical trees tends to the distribution of an unconditional critical tree. Hence,
%
%
$$ R(n) \leq \lim _{m \to \infty} \E_{\T_{n,m}} \big [ \Reff\big ((o,0) \lra \Phi(V_n)\big ) \big ] \, ,$$
where we bounded the resistance to generation $n$ by the resistance to a single vertex $\Phi(V_n)$. Therefore,
$$ R(n) \leq \sup_{m \geq 2n} \E_{\T_{n,m}} \big [ \Reff\big ((o,0) \lra \Phi(V_n)\big ) \big ] = \sum_{x \in \Z^d} \p^n(o,x) \gamma(n,x) \, .$$
By Theorem \ref{mainthm} we have
$$ R(n) \leq A n^{1-\alpha} \sum_{x : \|x\| \leq \sqrt{n}} \p^n(o,x)  + An^{1-\alpha} \sum_{x: \|x\| \geq \sqrt{n}} \p^n(o,x)  \Big ({\|x\|^2 \over n}\Big )^{\alpha} \, .$$
The first sum is bounded by $An^{1-\alpha}$. For the second sum we bound by
$$ \sum_{x} \p^n(o,x)  \Big ({\|x\|^2 \over n}\Big ) = 1 \, ,$$
(see Section \ref{ssec:rw-estimates}) concluding the proof. \qed
%

\subsection{Some random walk estimates}
\label{ssec:rw-estimates}
We provide here some standard random walk estimates that will be useful throughout the proof.
We denote by $\{ S(n) \}_{n \ge 0}$ a random walk with step distribution $\p^1$
and $S(0) = o$. Let $(S^1(n),\dots,S^d(n))$
denote the coordinates of $S(n)$ in a coordinate system that
diagonalizes $Q^{-1}$ (lower indices will be used for
the Euclidean coordinates). Due to independent and mean zero increments
and the definition of the norm, we have
\eqnst
{ \E \Big[ \| S(n) \|^2 \Big]
  = n \E \Big[ \| S(1) \|^2 \Big]
  = \frac{n}{d} \sum_{i,j=1}^d \E \Big[ S(1)_i Q^{-1}_{ij} S(1)_j \Big]
  = n \, .}
Applying Chebyshev's inequality, we get
\eqn{e:simple-lb}
{ \sum_{x \in \Z^d : \| x \| \le \sqrt{2}} \ \p^1(o,x)
  = 1 - \prob ( \| S(1) \|^2 > 2 )
  \ge 1/2\, .}
The central limit theorem \cite[Theorem 2.9.6]{D} implies that for any
$0 < L < \infty$ and any $v \in \R^d$, we have
\eqn{e:CLT}
{ \lim_{n \to \infty} \prob ( \| S(n) - \sqrt{n} v \| \le L \sqrt{n} )
  = C(v,d,L) \in (0, 1) \, ,}
with the constant $C(v,d,L)$ independent of $\p^1$.
%

The following proposition summarizes some estimates we will need
on the random walk $S$ conditioned on the event $\{ S(n) = x \}$.

\begin{proposition}
\label{prop:rw-estimates}
There exists $k_1 = k_1(\p^1)$, $C > 0$ and $\delta_1 = \delta_1(d) > 0$ such that the
following hold.\\
(i) Whenever $k_1 \le k \le n$, $\| x \| \le 4 n / \sqrt{k}$, we have
\eqn{e:sq-mean}
{ \E \Big[ \| S(k) \|^2 \,\Big|\, S(n) = x \Big]
  \le C k \, . }
(ii) Whenever $k_1 \le k \le \delta_1 n$, $\| x \| \le 4 n / \sqrt{k}$, we have
\eqn{e:sq-mean-cond}
{ \E \Big[ \| S(k) \|^2 \,\Big|\, S(n) = x,\, \| S(k) \| > \sqrt{k} \Big]
  \le C k \, .}
(iii) Whenever $k_1 \le k \le \delta_1 n$, $k_1 \le k' \le n-k$ and
$\| x \| \le \min\{ 4 n / \sqrt{k}, 4 n / \sqrt{k'} \}$, we have
\eqn{e:sq-mean-cond-other}
{ \E \Big[ \| S(k + k') - S(k) \|^2 \,\Big|\, S(n) = x,\, \| S(k) \| > \sqrt{k} \Big]
  \le C k' \, .}
\end{proposition}

For the proof of this proposition, we will use the exponential
moment assumption from Section \ref{ssec:IIBRW}.
Let $b_1 > 0$ be such that when $\| \beta \| < b_1$ we have
\eqnst
{ Z_\beta
  := \sum_{y \in \Z^d} e^{\beta \cdot y} \p^1(o,y)
  < \infty \, ,}
where $\cdot$ in the exponent denotes inner product with respect to
the quadratic form $\sum_{i,j} x_i Q^{-1}_{ij} y_j$.
Define the exponentially tilted step distribution
\eqnst
{ \p^1_\beta(o,y)
  = \frac{1}{Z_\beta} e^{\beta \cdot y} \p^1(o,y) }
Let $X(1), \dots, X(n)$ be i.i.d.~distributed according to $\p^1$, so that
$S(n) = X(1) + \dots + X(n)$, and let $X_\beta(1), \dots, X_\beta(n)$ be
i.i.d. distributed according to $\p^1_\beta$. Let
$S_\beta(n) = X_\beta(1) + \dots + X_\beta(n)$ and let
\eqnst
{ m_\beta
  := \E [ X_\beta(1) ]
  = \frac{1}{Z_\beta} \sum_{y \in \Z^d} y e^{\beta \cdot y} \p^1(o,y)
  = \nabla \log Z_\beta \, .}
Since the Jacobian of $\beta \mapsto m_\beta$ at $\beta = 0$ is
non-singular, for $v \in \R^d$ sufficiently close to $0$
there exists a unique $\beta$ such that $m_\beta = v$.
We write $Q_\beta$ for the covariance matrix of $X_\beta(1)$,
$D_\beta = \det(Q_\beta)^{1/2d}$, and $\| \cdot \|_\beta$ for the
norm arising from $Q_\beta^{-1}$. Note that $D_\beta$ and $\| \cdot \|_\beta$
depend continuously on $\beta$ in a neighbourhood of $0$.
In particular, for $\beta$ in a neighbourhood of $0$ we have
\eqn{e:norms-bnd}
{ \| v \|_\beta
  \le 2 \| v \| \, .}
Since $\E [ \| X(1) \|^2 ] = 1$,
for $\beta$ sufficiently close to $0$, we have
\eqn{e:Var-bnds}
{ \Sigma_\beta^2
  := \E \Big[ \| X_\beta(1) - m_\beta \|^2 \Big]
  \le 2 \, .}

We will need the following local limit theorem that is
uniform in small $\beta$.

\begin{lemma}
\label{lem:beta-LCLT}
There exists $C = C(d)$ and $0 < b_2 = b_2(\p^1) < b_1$ such that
the following hold.\\
(i) There exists $n_1 = n_1(\p^1)$ such that for all $y \in \Z^d$ we have
\eqn{e:beta-LCLT-ub}
{ \prob ( S_\beta(n) = y )
  \le \frac{2 C}{D_\beta^d n^{d/2}} \, ,}
when $n \ge n_1$, $\| \beta \| \le b_2$.\\
(ii) For any $0 < \eps < 1$ and $0 < L < \infty$ there exists
$n_2 = n_2(\p^1,\eps,L)$ such that for all $y \in \Z^d$
such that $\| y - n m_\beta \| \le L \sqrt{n}$ we have
\eqnspl{e:beta-LCLT}
{ \prob ( S_\beta(n) = y )
  &\le \frac{C(1 + \eps)}{D_\beta^d n^{d/2}} e^{ - d \| y - n m_\beta \|_\beta^2 / (2 n)}\, , \\
  \prob ( S_\beta(n) = y )
  &\ge \frac{C(1 - \eps)}{D_\beta^d n^{d/2}} e^{ - d \| y - n m_\beta \|_\beta^2 / (2 n)}\, .}
when $n \ge n_2$, $\| \beta \| \le b_2$.
\end{lemma}

We assumed above that the walk has period $1$. Trivial modifications can be made to handle the
case of period $2$, and we will not make this explicit in our arguments. \\

\noindent
{\bf Proof of Lemma \ref{lem:beta-LCLT}.}
The lemma can be proved by appealing to a local central limit theorem
for lattice distributions \cite[Theorem 2.5.2]{D}.
Note that the standard proof in \cite{D} can be followed, and this
gives uniformity in $\beta$.
\qed\\

Specializing to $\beta = 0$, we denote
\eqnst
{ D
  := \det(Q)^{1/2d} \, .}
Observe that with the norm introduced in \eqref{e:norm}, we have
\eqn{e:volume}
{ \sum_{x : \| x \| \le L} 1
  \geq  c D^d L^d \, ,}
for some constant $c > 0$ and all $L \geq 1$.
When $d \ge 3$,
the Green function $G(x) := \sum_{n=0}^\infty \p^n(o,x)$ satisfies (see \cite[Theorem 4.3.5]{LL}):
\eqnspl{e:Green}
{ G(x)
  \le \frac{C(d)}{D^d} \| x \|^{2-d}, \quad \text{when $\| x \| \ge L_1 = L_1(\p^1)$} .}


It follows from Lemma \ref{lem:beta-LCLT} that
there exist $0 < b_3 = b_3(\p^1) < b_1$, $k_2 = k_2(\p^1)$ and
$c = c(d) < 1$ such that for $k \ge k_2$ and $\| \beta \| \le b_3$ we have
\eqn{e:CLT-lb}
{ \prob ( \| S_\beta(k) - k m_\beta \| \le \sqrt{k} )
  \le c \, .}
We now choose $0 < b_0 = b_0(\p^1) \le \min \{ b_2, b_3 \}$
so that Lemma \ref{lem:beta-LCLT}, \eqref{e:norms-bnd},
\eqref{e:Var-bnds} and \eqref{e:CLT-lb} all
hold when $\| \beta \| \le b_0$. We also choose now $r_0 = r_0(\p^1) > 0$
such that $\| v \| \le r_0$ implies
$\| \beta \| \le b_0$ for the the unique $\beta$ such that
$m_\beta = v$. The constants $b_0$ and $r_0$ will now be fixed for the remainer
of the paper.

We are ready to prove Proposition \ref{prop:rw-estimates}.\\

\noindent
{\bf Proof of Proposition \ref{prop:rw-estimates}.}
Choose $k_1 = k_1(\p^1)$ in such a way that $4 / \sqrt{k_1} < r_0$
and $k_1 \ge k_2(\p^1)$ for the constant $k_2$ of \eqref{e:CLT-lb}.
We also require that $k_1 \ge n_1$ and $k_1 \ge n_2(\eps = 1/2, L = 1)$
for the constants $n_1$, $n_2$ from Lemma \ref{lem:beta-LCLT}.
Fix $x$, $n$ and $k$, and let $\beta$ be such that $m_\beta = x/n$.
Note that the choice of $k_1$ and the conditions on $x$ and $n$
imply that $\| \beta \| \le b_0$.

It is easy to check that conditional on $S(n) = x$,
the joint distribution of $X(1), \dots, X(n)$ is the same as the joint
distribution of $X_\beta(1), \dots, X_\beta(n)$ conditioned on
$S_\beta(n) = x$. Consequently, the joint distribution of
$X(1) - x/n, \dots, X(n) - x/n$, given $S(n) - x = 0$
is the same as the joint distribution of
$X_\beta(1) - m_\beta, \dots, X_\beta(n) - m_\beta$,
given $S_\beta(n) - n m_\beta = 0$.
Therefore, since $\E [ S(k) \,|\, S(n) = x ] = (k/n)x$, we have
\eqnspl{e:sq-mean-bnd}
{ \E \Big[ \| S(k) \|^2 \,\Big|\, S(n) = x \Big]
  &= \E \Big[ \Big\| S(k) - \frac{k}{n} x + \frac{k}{n} x \Big\|^2 \,\Big|\, S(n) = x \Big] \\
  &= \E \Big[ \Big\| S(k) - \frac{k}{n} x \Big\|^2 \,\Big|\, S(n) = x \Big]
    + \frac{k^2}{n^2} \| x \|^2 \, .}
The second term on the right hand side is at most $16 k$, by
our assumption on $\| x \|$. The first term on the right hand side of
\eqref{e:sq-mean-bnd} equals
\eqnst
{ \E \Big[ \| S_\beta(k) - k m_\beta \|^2 \,\Big|\, S_\beta(n) = n m_\beta \Big] \, .}
Conditional on $S_\beta(n) = n m_\beta$, the variables
$X_\beta(1), \dots, X_\beta(n)$ are exchangeable, and it is easy
to use $S_\beta(n) - n m_\beta = 0$ (expanding
the variance) that $X^j_\beta(k_1)$ and $X^j_\beta(k_2)$ are negatively
correlated for all $1 \le k_1 < k_2 \le n$ and $j = 1, \dots, d$.
It follows that
\eqnst
{ \E \Big[ \| S_\beta(k) - k m_\beta \|^2 \,\Big|\, S_\beta(n) = n m_\beta \Big]
  \le k \E \Big[ \| X_\beta(1) - m_\beta \|^2 \,\Big|\, S_\beta(n) = n m_\beta \Big] \, .}
It remains to estimate the conditional expectation on the right hand side.
Using Lemma \ref{lem:beta-LCLT}, this is at most
\eqnsplst
{ \sum_{y \in \Z^d} \| y - m_\beta \|^2 \p_\beta^1(o,y) \frac{\p_\beta^{n-1}(y,x)}{\p_\beta^n(o,x)}
  &\le \sum_{y \in \Z^d} \| y - m_\beta \|^2 \p_\beta^1(o,y) \frac{4 (n-1)^{d/2}}{n^{d/2}} \\
  &\le C \E \Big[ \| X_\beta(1) - m_\beta \|^2 \Big]
  \le C \Sigma_\beta^2 \, .}
By \eqref{e:Var-bnds}, we obtain the first statement
\eqref{e:sq-mean} of the proposition.

In order to prove \eqref{e:sq-mean-cond} it is sufficient, due to the just
proven part (i), to show that $\delta_1$ can be chosen
such that $\prob ( \| S(k) \| > \sqrt{k} \,|\, S(n) = x ) \ge c' > 0$.
For this, let $c$ be the constant in \eqref{e:CLT-lb} and recall
the constant $C_1$ in \eqref{e:norms-bnd}.
Observe that if $\| y \| \le \sqrt{k}$, we have
\eqnst
{ \| (x - y) - (n - k) m_\beta \|
  = \| y - k m_\beta \|
  \le \| y \| + k \frac{\| x \|}{n}
  \le \sqrt{k} + k \frac{4n}{\sqrt{k} n}
  = 5 \sqrt{k} \, .}
Hence due to \eqref{e:norms-bnd}, $\| y - k m_\beta \|_\beta^2 \le 100 k$.
We now use Lemma \ref{lem:beta-LCLT}(ii)
with $\eps > 0$ satisfying $c (1+\eps)/(1-\eps) < (1+c)/2$. We write
\eqnsplst
{ &\prob ( \| S(k) \| \le \sqrt{k} \,|\, S(n) = x ) \\
  &\qquad = \sum_{y : \| y \| \le \sqrt{k}} \p_\beta^k(o,y) \frac{\p_\beta^{n-k}(y,x)}{\p_\beta^{n}(o,x)} \\
  &\qquad \le \sum_{y : \| y \| \le \sqrt{k}} \p_\beta^k(o,y) \frac{(1+\eps) e^{- 50 d k / (n-k)} n^{d/2}}{(1-\eps) (n-k)^{d/2}} \\
  &\qquad \le c \frac{1+\eps}{1-\eps} e^{- 50 d \delta_1 / (1 - \delta_1)} \left( 1 + \delta_1 \right)^{d/2} \\
  &\qquad \le \frac{1+c}{2} e^{-50 d \delta_1 / (1 - \delta_1)} ( 1 + \delta_1 )^{d/2} \, .}
We choose $\delta_1 = \delta_1(d)$ so that the right hand side is $< 1$,
and this proves part (ii) of the proposition.

The last statement \eqref{e:sq-mean-cond-other} now follows easily.
Due to exchangability, and part (i), we have
\eqnst
{ \E \Big[ \| S(k+k') - S(k) \|^2 \,\Big|\, S(n) = x \Big]
  = \E \Big[ \| S(k') \|^2 \,\Big|\, S(n) = x \Big]
  \le C k' \, .}
Hence the statement follows from
$\prob ( \| S(k) \| > \sqrt{k} \,|\, S(n) = x ) \ge c' > 0$
proved in part (ii).
\qed

\section{{\bf Setting up the induction scheme}}


We begin by introducing some useful notation. Given an instance of $\T_{n,m}$, consider
some small $\delta > 0$, where we assume that $\delta n$ is an integer. We write
$$ X_i = V_{i \delta n} \qquad i=0,1,\ldots, \lf \delta^{-1} \rf \, ,$$
and write $x_i \in \Z^d$, $i=0,1,\ldots, \lf \delta^{-1} \rf$ for the random spatial location
of $X_i$, that is, $x_i$ is the unique vertex satisfying $\Phi(X_i) = (x_i, i\delta n)$.
Write $\T_{n,m}(\ell)$ for the subtree of $\T_{n,m}$ emanating from $V_\ell$ off the backbone (including the vertex $V_\ell$). \\

We let $N = \lf (K \delta)^{-1} \rf - 1$, and subdivide the backbone into
$N$ stretches of length $K \delta n$, and a remaining part of length
at least $K \delta n$ and less than $2 K \delta n$.

We begin with some definitions that are depicted in Figure \ref{fig:K-tree-good}.

\begin{definition} \label{udp} For $\ell$ satisfying $i \delta n \leq \ell < (i+1)\delta n$ we say that a backbone vertex $V_\ell$ has the {\em unique descendant property} (UDP) if in $\T_{n,m}(\ell)$ it has a unique descendant at level $(i+1)\delta n$ that reaches level $(i+2)\delta n$. For any other vertex $V$ of $\T_{n,m}$ at level $i\delta n$ we say that $V$ has UDP if it has a unique descendant at level $(i+1)\delta n$ that reaches level $(i+2)\delta n$.
\end{definition}

%


\begin{definition}
\label{ktreegood}
Given an integer $K \geq 1$, a number $\delta>0$ such that $K \delta \le (1/2)$
and an instance of $\T_{n,m}$ we say that a
sequence $(i, i+1, \ldots, i+K)$ of length $K+1$ is $K$-{\em tree-good} if the following holds:
\begin{enumerate}
\item[(1)] There exists a unique $i \delta n \le \ell_1 < (i+1) \delta n$ such
that $\T_{n,m}(\ell_1)$ reaches height $(i+2)\delta n$. Moreover,
this unique $\ell_1$ satisfies $(i+1/4) \delta n \leq \ell_1 \leq (i+3/4)\delta n$.
%
%
%
%
\item[(2)] $V_{\ell_1}$ has UDP and we call the unique descendant $\YY_{i+1}$.
For all $i'$ satisfying $i+2 \leq i' \leq i+K$ we inductively define the vertices $\YY_{i'}$ of
$\T_{n,m}(\ell_1)$ as follows. We require that $\YY_{i'-1}$ has UDP
and call the unique descendant $\YY_{i'}$.
%
%
\item[(3)] There exists a unique $(i+K-1)\delta n \le \ell_2 < (i+K) \delta n$ such that
$\T_{n,m}(\ell_2)$ reaches height $(i+K+1)\delta n$. Moreover, this unique $\ell_2$ satisfies
$(i+K-3/4) \delta n \leq \ell_2 \leq (i+K-1/4) \delta n$.
%
%
\item[(4)] $V_{\ell_2}$ has UDP, and we call the unique descendant
$\XX'_{i+K}$. The vertex $\XX'_{i+K}$ has UDP,
and we call the unique descendant $\XX'_{i+K+1}$.
Similarly, $\YY_{i+K}$ has UDP,
and we call the unique descendant $\YY_{i+K+1}$.
\end{enumerate}
\end{definition}

\psfrag{Xi}{$\scrs{X_i}$}
\psfrag{Vell1}{$\scrs{V_{\ell_1}}$}
\psfrag{Xi+1}{$\scrs{X_{i+1}}$}
\psfrag{Xi+2}{$\scrs{X_{i+2}}$}
\psfrag{Xi+K-1}{$\scrs{X_{i+K-1}}$}
\psfrag{Vell2}{$\scrs{V_{\ell_2}}$}
\psfrag{Xi+K}{$\scrs{X_{i+K}}$}
\psfrag{X'i+K}{$\scrs{\XX'_{i+K}}$}
\psfrag{X'i+K+1}{$\scrs{\XX'_{i+K+1}}$}
\psfrag{Yi+1}{$\scrs{\YY_{i+1}}$}
\psfrag{Yi+2}{$\scrs{\YY_{i+2}}$}
\psfrag{Yi+K-1}{$\scrs{\YY_{i+K-1}}$}
\psfrag{Yi+K}{$\scrs{\YY_{i+K}}$}
\psfrag{Yi+K+1}{$\scrs{\YY_{i+K+1}}$}
\psfrag{dn}{$\scrs{\delta n}$}

\psfrag{xi}{$\scrs{x_i}$}
\psfrag{vell1}{$\scrs{v_{\ell_1}}$}
\psfrag{xi+1}{$\scrs{x_{i+1}}$}
\psfrag{xi+2}{$\scrs{x_{i+2}}$}
\psfrag{xi+K-1}{$\scrs{x_{i+K-1}}$}
\psfrag{vell2}{$\scrs{v_{\ell_2}}$}
\psfrag{xi+K}{$\scrs{x_{i+K}}$}
\psfrag{x'i+K}{$\scrs{x'_{i+K}}$}
\psfrag{x'i+K+1}{$\scrs{x'_{i+K+1}}$}
\psfrag{yi+1}{$\scrs{y_{i+1}}$}
\psfrag{yi+2}{$\scrs{y_{i+2}}$}
\psfrag{yi+K-1}{$\scrs{y_{i+K-1}}$}
\psfrag{yi+K}{$\scrs{y_{i+K}}$}
\psfrag{yi+K+1}{$\scrs{y_{i+K+1}}$}
\psfrag{<=sqdn}{$\scrs{\le \sqrt{\delta n}}$}

\begin{figure}
\includegraphics[scale=0.6]{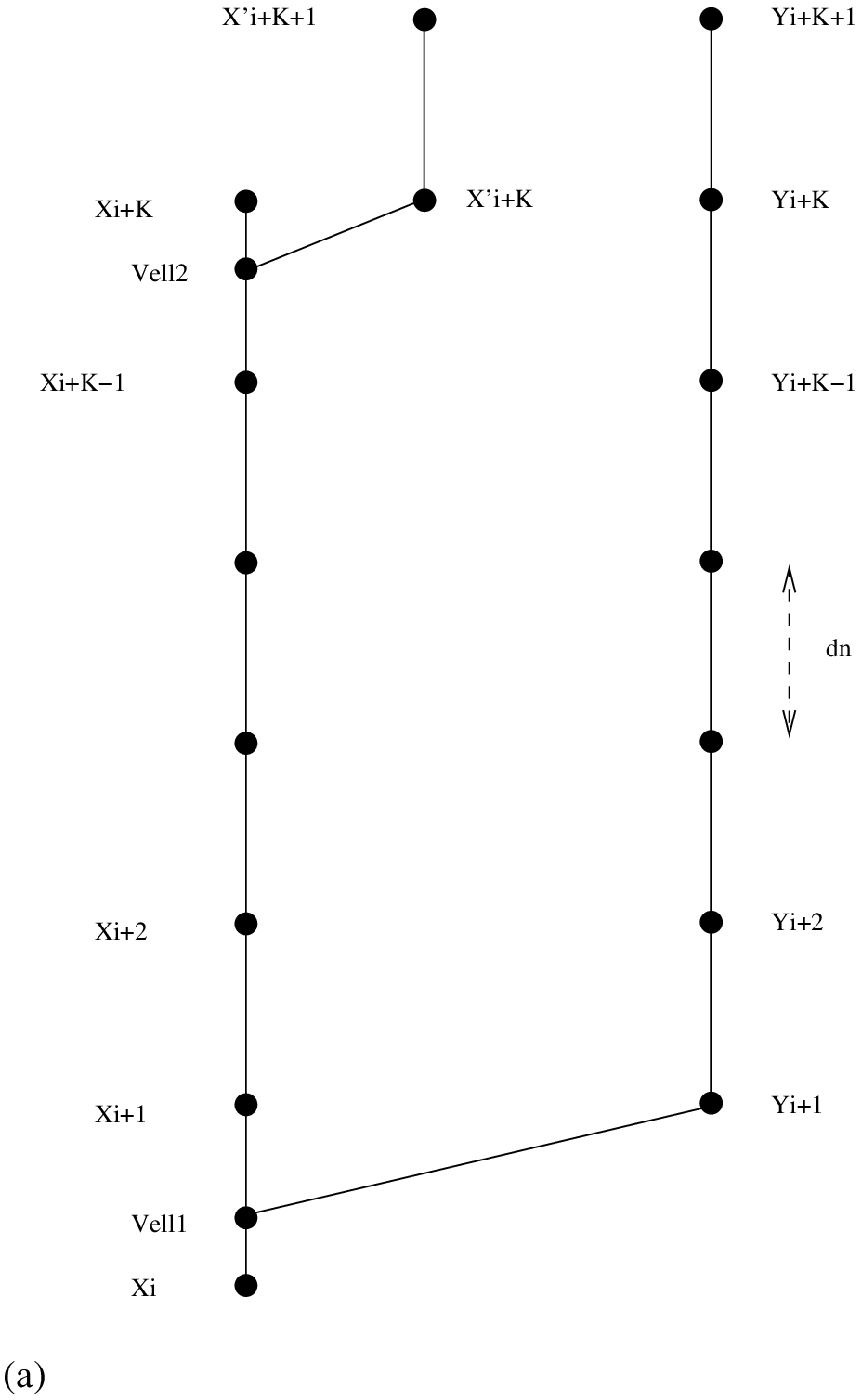}
\hskip1cm
\includegraphics[scale=0.6]{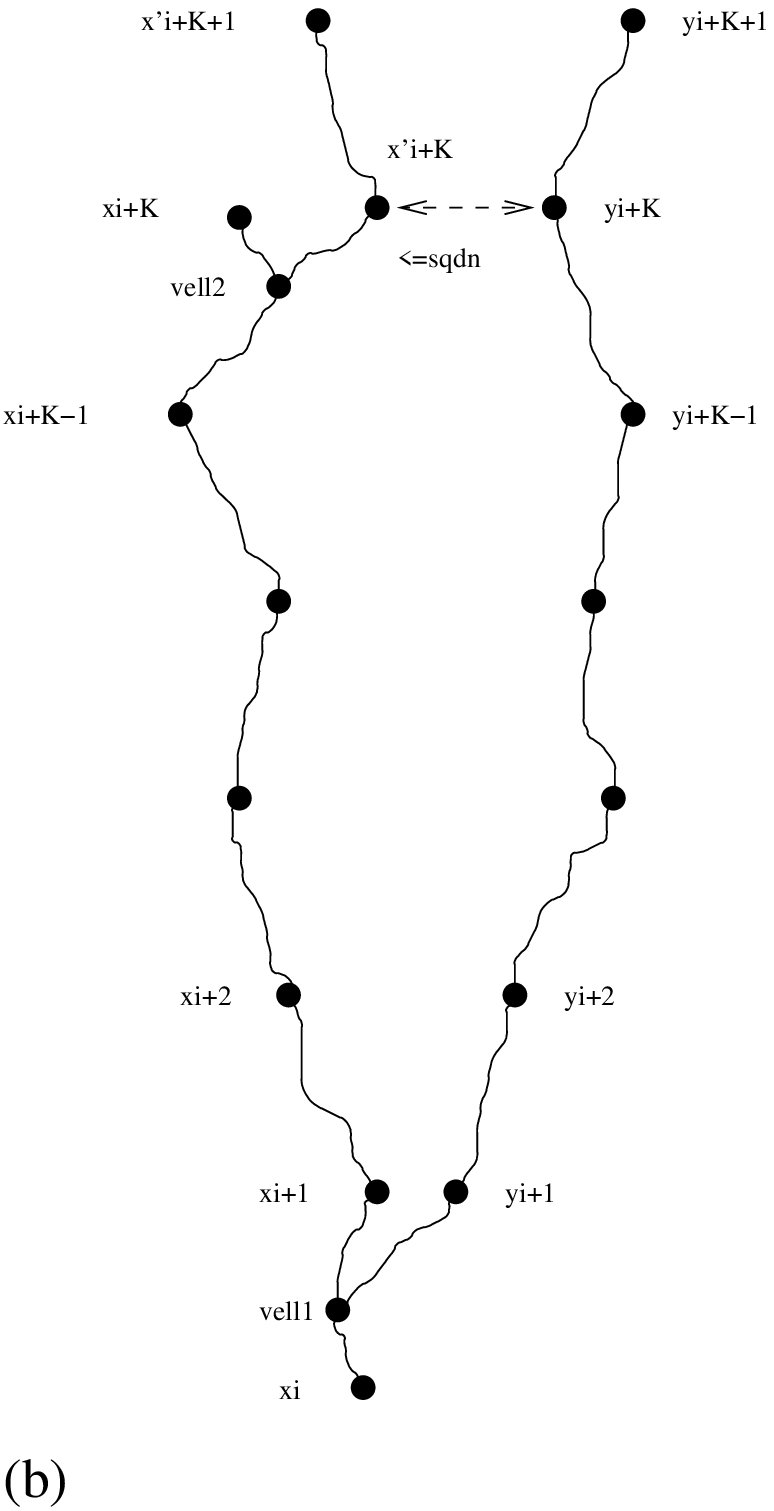}
\caption{(a) Illustration of $K$-tree-good.
(b) Illustration of $K$-spatially-good. All spatial distances
between consecutive vertices are at most
$\sqrt{\text{time difference}}$ and
the spatial distance between $x'_{i+K}$ and $y_{i+K}$
is at most $\sqrt{\delta n}$.}
\label{fig:K-tree-good}
\end{figure}

Given a $K$-tree-good sequence $(i, \ldots, i+K)$
we denote by $\V_{\ell_1}^+$ (respectively $\V_{\ell_2}^+$)
the child of $V_{\ell_1}$ (respectively $V_{\ell_2}$)
leading to $\YY_{i+1}$ (respectively $\XX'_{i+K}$).
We further define the spatial locations $y_{i'}$ by
$\Phi(\YY_{i'}) = (y_{i'}, i' \delta n)$ for
$i+1 \leq i' \leq i+K$, and
we similarly define $x'_{i+K}$, $x'_{i+K+1}$,
$v_{\ell_1}$, $v_{\ell_1}^+$, $v_{\ell_2}$, $v_{\ell_2}^+$.

We will write $U \prec W$ to denote that $W$ is a descendant of $U$, and write $h(U), h(W)$ for their respective heights in the tree (in particular, $h(W)>h(U)$).

\begin{definition}
\label{typicallyspaced}
Let $U \prec W$ be two tree vertices and let $u,w\in\Z^d$ be defined by
$\Phi(U)=(u,h(U))$ and $\Phi(W)=(w,h(W))$. We say that $U$ and $W$ are
{\em typically-spaced} if $\| w - u \| \leq \sqrt{h(W)-h(U)}$.
Denote this event by $\TS(U,W)$.
\end{definition}

\begin{definition}
\label{kspatialgood}
We say that a $K$-tree-good sequence $(i,\ldots, i+K)$ is
\emph{$K$-spatially-good} if the following holds.
\begin{enumerate}
\item[(5)] \begin{itemize}  \item $\TS(X_i, V_{\ell_1})$,
                            \item $\TS(V_{\ell_1+1}, X_{i+1})$,
                            \item For each $i+1 \leq j \leq i+K-2$ we have $\TS(X_j, X_{j+1})$,
                            \item $\TS(X_{i+K-1}, V_{\ell_2})$,
                            \item $\TS(V_{\ell_2+1}, X_{i+K})$,
            \end{itemize}
\item[(6)] \begin{itemize}  \item $\TS(V_{\ell_1}^+, \YY_{i+1})$,
                            \item For each $i+1 \leq j \leq i+K-1$ we have $\TS(\YY_j, \YY_{j+1})$,
                            \item $\TS(V_{\ell_2}^+, \XX'_{i+K})$,
                            \item $\| x'_{i+K} - y_{i+K} \| \le \sqrt{\delta n}$.
            \end{itemize}
\end{enumerate}
\end{definition}

\begin{definition} \label{def:kgood} When a sequence $(i, \ldots, i+K)$ is both $K$-tree-good
and $K$-spatially-good we say that it is $K$-{\em good}. Let $\A(i)$ be the
event that $(i, \ldots, i+K)$ is $K$-good.
\end{definition}

Next, let $(i, \ldots, i+K)$ be a $K$-good sequence and let $U_1, U_2$  be two vertices at the same height such
that $U_1 \succ \XX'_{i+K}$ and $U_2 \succ \YY_{i+K}$. Given these, we write $Z_1$ for the highest common ancestor
of $U_1$ and $\XX'_{i+K+1}$ and $Z_2$ for the highest common ancestor of $U_2$ and $\YY_{i+K+1}$
(see Figure \ref{fig:intersect}). Further, we denote by $Z_1^+$ (respectively $Z_2^+$)
the child of $Z_1$ (respectively $Z_2$) leading to $U_1$ (respectively $U_2$).

\psfrag{Phi(X'i+K)}{$\scrs{\Phi(\XX'_{i+K})}$}
\psfrag{Phi(Yi+K)}{$\scrs{\Phi(\YY_{i+K})}$}
\psfrag{Phi(X'i+K+1)}{$\scrs{\Phi(\XX'_{i+K+1})}$}
\psfrag{Phi(Yi+K+1)}{$\scrs{\Phi(\YY_{i+K+1})}$}
\psfrag{Phi(U_1) = Phi(U_2)}{$\scrs{\Phi(U_1) = \Phi(U_2)}$}
\psfrag{Phi(Z_1)}{$\scrs{\Phi(Z_1)}$}
\psfrag{Phi(Z_2)}{$\scrs{\Phi(Z_2)}$}
\psfrag{Phi(Z+_1)}{$\scrs{\Phi(Z^+_1)}$}
\psfrag{Phi(Z+_2)}{$\scrs{\Phi(Z^+_2)}$}

\begin{figure}
\includegraphics[scale=0.6]{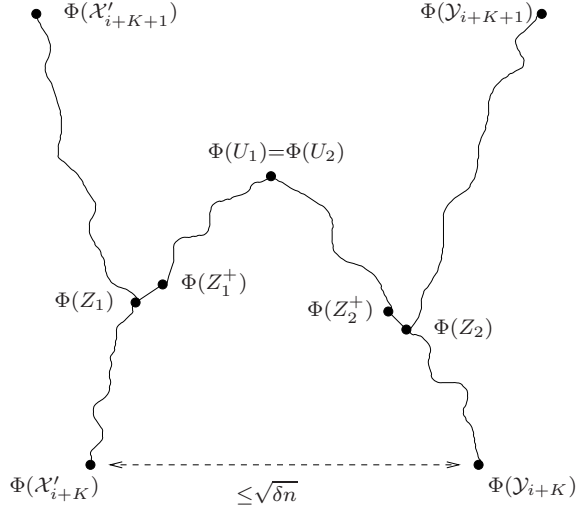}
\caption{The labelling of vertices in the two (potentially)
intersecting trees emanating from $\XX'_{i+K}$ and $\YY_{i+K}$.}
\label{fig:intersect}
\end{figure}

\begin{definition}
\label{intersect} We say that $U_1, U_2$ \emph{intersect-well} if the following conditions hold:
\begin{itemize}
\item[1.] $U_1 \succ \XX'_{i+K}$, $U_2 \succ \YY_{i+K}$,
\item[2.] $(i+K+(5/6)) \delta n \le h(U_1) = h(U_2) \le (i+K+1) \delta n$;
\item[3.] $(i+K+(1/2)) \delta n \le h(Z_1), h(Z_2) \le (i+K+(4/6)) \delta n$;
\item[4.] $\TS(\XX'_{i+K}, Z_1)$, $\TS(Z_1^+, U_1)$, $\TS(\YY_{i+K}, Z_2)$, $\TS(Z_2^+, U_2)$;
\item[5.] $\Phi(U_1) = \Phi(U_2)$.
\end{itemize}
And define the random set $\I$ by
\begin{equation}
\label{e:I-defn}
\begin{split}
  \I  = \left\{ (U_1,U_2) : \text{$U_1$ and $U_2$ intersect-well} \right\}\, ,
\end{split}
\end{equation}
Lastly, we define the event $\B(i,c_0)$ where $c_0>0$ is a constant
$$\B(i,c_0) = \A(i) \cap \big \{ |\I| \geq \frac{c_0 \sigma^4}{D^d} (\delta n)^{(6-d)/2} \big \} \, .$$


%
\end{definition}

Our first theorem is that $K$-good runs $(i, i+1, \ldots, i+K)$ occur with positive density and in each,
the probability of seeing many intersections occurs with positive probability.


\begin{theorem} [Intersections exist]
\label{goodthingshappen}
Assume that $d = 5$. There exist constants $c_0,c_1>0$ and for any $K \geq 2$ there exists $c_2 = c_2(K) > 0$,
and $n_3 = n_3(\sigma^2, C_3, \p^1, K)$ such that for any $0 < \delta < (K+4)^{-1}$, whenever $\delta n \ge n_3$ and
$x$ satisfies $\| x \| \leq \sqrt{2n / \delta}$, 
we have
$$ \prob( \A(i) \, \mid \, \Phi(V_n) = (x,n) ) \geq c_2 \, .$$
and
$$ \prob ( \B(i,c_0) \, \mid \, \A(i)\, , \Phi(V_n) = (x,n) ) \geq c_1$$
for $i = 0, K, 2K, \dots, (N-1)K$.
\end{theorem}

To proceed let us define
$$ \gamma(n) = \sup_{x : \| x \| \leq \sqrt{n}} \gamma(n,x) \, .$$
When all the good events occur, it is immediate by definition and the triangle inequality \eqref{triangle} that the resistance between $X_i$ and $X_{i+K}$
is bounded above by $K \gamma(\delta n)$. The following theorem shows that the intersections create a ``short-cut'' in the electric circuit,
allowing us to bound the resistance between the two ends of the the run $(i, \ldots, i+K)$ using
the parallel law of electric resistance \eqref{parallel} essentially by ${3 \over 4} K \gamma(\delta n)$. This multiplicative constant improvement allows
the induction argument to work.

\begin{theorem} [Analysis of good blocks]
\label{goodblock}
There exists $K_0 < \infty$ and $n_4 = n_4(\sigma^2, C_3, \p^1)$ such that if $K \geq K_0$ and $\delta n \ge n_4$, we have
$$ \E \Big [ \Reff(\Phi(X_i) \lra \Phi(X_{i+K})) \, \mid \, \A(i), \B(i,c_0), \Phi(V_n) = (x,n) \Big ]
\leq {3 K \over 4} \max_{1 \le k \le \delta n} \gamma(k) $$
for $i = 0, K, 2K, \dots, \lf (K \delta)^{-1} \rf - 2$.
\end{theorem}

To complete the induction step we also need a bound on the resistance conditioned on
$\A(i)^c \cup \B(i,c_0)^c$. This is rather lengthy, since for each reason that either
$\A(i)$ or $\B(i,c_0)$ fail, we provide a different bound on the resistance which we
eventually collect together at the proof of the induction step.

\subsection{Organization}
The proof of Theorem \ref{goodthingshappen} is done in Section \ref{sec:goodthingshappen} and the proof of Theorem \ref{goodblock} is presented in Section \ref{sec:goodblocks}. The analysis of the resistance when the good events fail to occur is presented in Sections \ref{sec-badtree} and \ref{sec-badspatial}.
%


\section{Existence of intersections}
\label{sec:goodthingshappen}

In this section we prove Theorem \ref{goodthingshappen}. In Section \ref{ssec:K-good} we show that
$K$-good runs $(i, \ldots, i+K)$ occur with positive probability, proving the first statement of
Theorem \ref{goodthingshappen}. In Section \ref{ssec:enough} we show that given a $K$-good run,
there are ``enough'' intersections with positive probability, proving the second statement of
Theorem \ref{goodthingshappen}.

\subsection{Preliminaries} Recall that
$\tilde{p}(k) = (k+1) p(k+1)$ for $k \geq 0$.
We denote by $\{ \NN_n \}_{n \ge 0}$ a branching process
with $\NN_0 = 1$ and progeny distribution $p(k)$, and
by $\{ \widetilde{\NN}_n \}_{n \ge 0}$ a branching
process with $\widetilde{\NN}_0 = 1$ and progeny
$\tilde{p}(k)$ in the first generation and
progeny $p(k)$ afterwards. Note that for all
$n \ge 1$ we have $\E \NN_n = 1$ and
\eqnst
{ \E \widetilde{\NN}_n
  = \sum_{k\geq 0} k \tilde{p}(k)
  = \sum_{k\geq 1} k(k-1) p(k)
  = \Var (\NN_1)
  = \sigma^2 \, . }
We denote by $f(s)$ and $\tilde{f}(s) = f'(s)$ the generating functions
of $p$ and $\tilde{p}$, respectively. Then the generating
functions of $\NN_n$ and $\widetilde{\NN}_n$ are
$f_n(s)$ and $g_n(s) := \tilde{f}(f_{n-1}(s))$, respectively,
where $f_n(s)$ is the $n$-fold composition of $f$ with
itself.

We denote by $\theta(n) = f_n(0) = \prob ( \NN_n > 0 )$ the survival
probability of the branching process up to time $n$.
It is well known \cite{Kolm,ANbook} that
\eqn{e:Kolm}
{ {\theta(n) \sigma^2 n \over 2} \to 1 \text{ as $n \to \infty$.} }
Furthermore, there exists $n_5 = n_5 (C_3) < \infty$
such that
\eqn{e:Kolm2}
{ \frac{1}{\sigma^2 n}
  \le \theta(n)
  \le \frac{3}{\sigma^2 n}, \qquad n \ge n_5 \, . }
Moreover, there exists $n'_5 = n'_5 (C_3)$ such that we have
\eqn{e:diff-theta}
{ \theta(n) - \theta(m)
  \ge \frac{1}{2 \sigma^2 n} \, , \quad \text{whenever $m \ge 2n$, $n \ge n_5'$.} }

\begin{lemma}
\label{lem:g_n'}
For any $C > 0$ there exists $c' = c'(C) > 0$
and $n_6 = n_6(C,\sigma^2,C_3) < \infty$ such that
for all $n \ge n_6$ we have
\eqnsplst
{ f_n' \left( 1 - \frac{C}{\sigma^2 n} \right)
  &\ge c' \\
  g_n' \left( 1 - \frac{C}{\sigma^2 n} \right)
  &\ge c' \sigma^2 \, .}
\end{lemma}

\begin{proof}
We have that
\eqn{e:f_n'}
{ f_n'(s)
  \ge \left( f'(s) \right)^n
  \ge \left( 1 - (1 - s) f''(1) \right)^n \, .}
Indeed, the first inequality follows by appealing to the chain rule and using the fact that $f_n(s) \geq s$ for $s\in[0,1]$ and that $f$ is convex. The second inequality follows from the mean-value theorem together with the fact that $f''$ is increasing (the coefficients of the Taylor series of $f$ are non-negative by definition). Substituting $s = 1 - C/\sigma^2 n$ gives the first statement (recall that $f''(1)=\sigma^2$).
For the second statement, observe that
\eqnsplst
{ g_n'(s)
  &= \tilde{f}'(f_{n-1}(s)) \, f_{n-1}'(s)
  \ge \tilde{f}'(s) \left( f'(s) \right)^{n-1}
  = f''(s) \, \left( f'(s) \right)^{n-1} \\
  &\ge \left( f''(1) - (1 - s) f'''(1) \right)
      \left( 1 - (1 - s) f''(1) \right)^{n-1} \, .}
Substituting $s = 1 - C/\sigma^2 n$ and using that $f'''(1) \leq C_3$ yields the result.
\end{proof}

\begin{lemma}
\label{survival}
There exist $n_7 = n_7(C_3) < \infty$ such that
\eqn{e:survival-uncond}
{ \frac{1}{2n}
  \le \prob ( \widetilde{\NN}_n > 0 )
  \le \frac{3}{n}, \quad \text{whenever $n \ge n_7$,} }
and
\eqn{e:survival-cond}
{ \frac{1}{4n}
  \le \prob ( \widetilde{\NN}_n > 0 \,|\, \widetilde{\NN}_m = 0 )
  \le \frac{6}{n}, \quad \text{whenever $n \ge n_7$, $m \ge 2n$.} }
\end{lemma}

\begin{proof}
For the upper bound in \eqref{e:survival-uncond},
if the process survives $n$ generations, one of
the particles at generation $1$ needs to survive $n-1$ generations,
so by \eqref{e:Kolm2} we bound this probability by
\eqnsplst
{ \sum_{k=1}^\infty \tilde{p}(k) k \frac{3}{\sigma^2 n}
  \leq \frac{3}{n} \, .}
For the lower bound in \eqref{e:survival-uncond},
we write
\eqnspl{e:uncond-lb}
{ \prob ( \widetilde{\NN}_n > 0 )
  &= 1 - \tilde{f}(f_{n-1}(0))
  = 1 - f'(1 - \theta(n-1)) \\
  &\ge 1 - f' \left( 1 - \frac{1}{\sigma^2 (n-1)} \right) \\
  &\ge \frac{1}{\sigma^2 (n-1)}
       f'' \left( 1 - \frac{1}{\sigma^2 (n-1)} \right) \, ,}
where the last inequality is due to the mean-value theorem. As before, $f''(s) \geq (f''(1) - (1 - s) f'''(1))$ and $f''(1)=\sigma^2$ and $f'''(1) \leq C_3$ gives the lower bound.

%

The proof of \eqref{e:survival-cond} is quite similar.
The upper bound follows easily, since by \eqref{e:survival-uncond}
we have for large enough $n$ the inequality
\eqnst
{ \prob ( \tilde{N}_n > 0 \,|\, \widetilde{\NN}_m = 0 )
  \le 2 \prob ( \widetilde{\NN}_n > 0 ) \, .}
For the lower bound, using \eqref{e:diff-theta} we write:
\eqnspl{e:cond-lb}
{ \prob ( \widetilde{\NN}_n > 0 \,|\, \widetilde{\NN}_m = 0 )
  &\ge \prob ( \widetilde{\NN}_n > 0,\, \widetilde{\NN}_m = 0 ) \\
  &= \tilde{f}(f_{m-1}(0)) - \tilde{f}(f_{n-1}(0)) \\
  &= f'(1 - \theta(m-1)) - f'(1 - \theta(n-1)) \\
  &\ge \left( \theta(n-1) - \theta(m-1) \right) f''( 1 - \theta(n-1) ) \\
  &\ge \frac{1}{2 \sigma^2 n} f'' \left( 1 - \frac{c}{\sigma^2 n} \right) \, .}
This is now bounded from below as in \eqref{e:uncond-lb}.
\end{proof}


\subsection{$K$-good runs occurs}
\label{ssec:K-good}

The proof is broken down into a series of lemmas showing that each of the conditions involved
in a run $(i,\ldots,i+K)$ being $K$-tree-good and $K$-spatially-good (that is,
the conditions in Definitions \ref{ktreegood} and \ref{kspatialgood}) holds with probability
bounded away from $0$. For $a = 1, \dots, 6$ let $\D_{(a)}$ denote the event that
condition (a) in Definitions
\ref{ktreegood} and \ref{kspatialgood} is satisfied.

We start by analyzing the conditions in Definition \ref{ktreegood}(1)--(4).
Recall that these only involve the branching process,
hence here the conditioning on $\{ \Phi(V_n) = (x,n) \}$
present in Theorem \ref{goodthingshappen} is irrelevant.
Therefore we omit it in the lemmas below.
%

\begin{lemma}
\label{lem:l1-exists}
There exists $c > 0$ such that we have $$\prob ( \D_{(1)} ) \ge c$$
whenever $\delta n \ge n_7$ .
\end{lemma}

\begin{proof} Assume without loss of generality that $i=0$ (the proof will be the same for any $0 \leq i \leq \delta^{-1}$).
For any $\ell$ satisfying $0 \leq \ell < \delta n$, let $\D_{(1)}(\ell)$ be the event
that the random tree attached to $V_{\ell}$, that is $\T_{n,m}(\ell)$,
reaches level $2 \delta n$. 
So $\D_{(1)}$ is the event that exactly one of the events
$\{ \D_{(1)}(\ell) \}$ occurs, and that the index $\ell_1$
of that event lies between $(1/4)\delta n$ and $(3/4) \delta n$.
The events $\{\D_{(1)}(\ell)\}$ are independent, and
due to \eqref{e:survival-cond} each has probability between
$\frac{1}{8 \delta n}$ and $\frac{6}{\delta n}$.
Hence $\prob ( \D_{(1)} ) \ge c > 0$.
\end{proof}

\begin{lemma}
\label{lem:Yi's exist}
There exists $c=c(K) > 0$ such that
we have $$\prob ( \D_{(2)} \,|\, \D_{(1)},\, \ell_1 ) \ge c > 0$$
whenever $\delta n \ge \max \{ n_5, n_5', 4 n_6(3,\sigma^2,C_3), n_7 \}$.
\end{lemma}

\begin{proof} Again we assume that $i=0$ (the reader will notice that we only use the fact that $m-i\delta n \geq n$). The probability that $V_{\ell_1}$ has UDP given $\D_{(1)}$ equals
$$ \frac{\sum_{k \ge 1} \prob ( \widetilde{\NN}_{\delta n - \ell_1} = k ) \,
    k \, ( \theta(\delta n) - \theta(m - \delta n) ) \,
    (1 - \theta( \delta n ))^{k-1}}
    {\prob( \widetilde{\NN}_{2 \delta n - \ell_1} > 0,\, \widetilde{\NN}_{m - \ell_1} = 0 )} \, ,$$
since $\T_{n,m}(\ell_1)$ is now conditioned to survive $2\delta n- \ell_1$ generations, but that the $m-\ell_1$-th generation died out.
Hence

\eqnspl{e:D(1)}
{ &\prob (\text{$V_{\ell_1}$ has UDP} \,|\, \D_{(1)},\, \ell_1 ) \geq \frac{\theta(\delta n) - \theta(m-\delta n)}{\prob( \widetilde{\NN}_{2 \delta n - \ell_1} > 0 )}
    g_{\delta n - \ell_1}'( 1 - \theta( \delta n ) ) \, .}
Due \eqref{e:diff-theta} and Lemmas \ref{survival} and \ref{lem:g_n'},
the right hand side of \eqref{e:D(1)} is at least a universal constant $c'>0$. %

Now, conditioned on $V_{\ell_1}$ having UDP, the descendant tree emanating from $\YY_1$ is a critical tree conditioned to survive $\delta n$ generations but not $m-\delta n$ generations. So the conditional probability that $\YY_1$ has UDP equals
$$ \frac{\sum_{k \ge 1} \prob ( \NN_{\delta n} = k ) \,
    k \, ( \theta(\delta n) - \theta(m - \delta n) ) \,
    (1 - \theta( \delta n ))^{k-1}}
    {\prob( \NN_{\delta n} > 0,\, \NN_{m - \delta n} = 0 )} \, ,$$
and similarly this is bounded below by $c'$. Iterating this argument over $\YY_2, \YY_3, \ldots \YY_{K-1}$ gives a probability of at least $c(K)=c'^K$, as required.
%
%
%
\end{proof}

\begin{lemma}
\label{lem:l2-exists}
There exists $c>0$ such that
$$\prob ( \D_{(3)} \,|\, \D_{(1)},\, \ell_1,\, \D_{(2)} ) = \prob ( \D_{(3)} ) \ge c$$
whenever $\delta n \ge n_7$.
\end{lemma}
\begin{proof} The proof is the same as the proof of Lemma \ref{lem:l1-exists}.
\end{proof}

\begin{lemma}
\label{lem:X_{i+K} etc}
There exists $c>0$ such that
$$\prob ( \D_{(4)} \,|\, \D_{(1)},\, \ell_1,\, \D_{(2)},\, \D_{(3)},\, \ell_2 ) \ge c$$
whenever $\delta n \ge \max \{ n_5, n_5', 4 n_6(3,\sigma^2,C_3), n_7 \}$.
\end{lemma}
\begin{proof}
This is proved almost identically to Lemma \ref{lem:Yi's exist}.
\end{proof}

We next show that the conditions in Definition \ref{kspatialgood}(5)--(6)
also hold with probability bounded away from $0$.

\begin{lemma}
\label{lem:show-K-spatial-good}
There exists $c=c(K) > 0$ and $n_8 = n_8(\p^1,K)$ such that whenever
$0 < \delta < 1/(K+4)$, $\delta n \ge n_8$ and
$\| x \| \leq \sqrt{2n / \delta}$, 
we have
\eqn{e:(5)-(6)}
{ \prob \big( \D_{(5)}, \D_{(6)} \,|\,
    \D_{(1)},\, \ell_1,\, \D_{(2)}, \D_{(3)},\,
    \ell_2,\, \D_{(4)},\, \Phi(V_n) = (x,n) \big)
  \ge c \, .}
\end{lemma}

\begin{proof} Let us condition on the entire branching process tree $\T_{n,m}$
in which $\D_{(1)}$--$\D_{(4)}$ hold. It will be convenient to consider the event $\D_{(5)}' \subset \D_{(5)}$ where we replace the requirements in Definition \ref{kspatialgood},(5) by
\begin{enumerate}
                            \item $\|x_i - v_{\ell_1}\| \leq (1/2) \sqrt{\ell_1 - i \delta n}$,
                            \item $\|v_{\ell_1+1} - x_{i+1} \| \leq (1/2) \sqrt{i \delta n - \ell_1 - 1}$,
                            \item For each $i+1 \leq j \leq i+K-2$ we have $\|x_j-x_{j+1}\|\leq (1/2) \sqrt{\delta n}$,
                            \item $\|x_{i+K-1}-v_{\ell_2}\|\leq (1/2) \sqrt{\ell_2 - (i+K-1)\delta n}$,
                            \item $\|x_{i+K} - v_{\ell_2+1}\| \leq (1/2) \sqrt{(i+K)\delta n - \ell_2 - 1}$.
                            \item $\| v_{\ell_1} - v_{\ell_1+1} \|,\, \| v_{\ell_2} - v_{\ell_2+1} \| \leq \sqrt{2}$
\end{enumerate}
We will show that
\be \label{e:anothersrwestimate} \prob( \D_{(5)}', \D_{(6)} | \T_{n,m} ) \geq c \, ,\ee
and that
\be\label{e:srwestimate} \prob( \Phi(V_n) = (x,n) \,|\, \D_{(5)}', \D_{(6)}, \T_{n,m} ) \geq c \prob( \Phi(V_n) = (x,n) | \T_{n,m} ) \, ,\ee
which will conclude our proof. To prove \eqref{e:anothersrwestimate} we first note that the events of $\D_{(5)}'$ are all independent and each occurs with probability bounded below by a constant, by \eqref{e:CLT} and \eqref{e:simple-lb}. Conditioned on $x_i, x_{\ell_1}, x_{\ell_1+1}, x_{i+1}, \dots, x_{i+K}$ that satisfy $\D_{(5)}'$, the event $\D_{(6)}$ has probability at least $c=c(K)>0$, indeed, because of the factors $1/2$ in the definition of $\D_{(5)}'$, repeated application of the central limit theorem yields that the displacement requirements in $D_{(6)}$ can be satisfied.

To prove \eqref{e:srwestimate} we condition on the value of $z = x_{i+K} - x_i$. Choose $n_8$ large enough
so that the conditions $\| x \| \le \sqrt{2n / \delta}$ and $\delta n \ge n_8$ imply $\| x/n \| \le r_0$
(where $r_0$ is the constant chosen in Section \ref{ssec:rw-estimates}).
Let $\beta$ be such that $m_\beta = x/n$.

Observe that
\eqnsplst
{ \| z - K \delta n m_\beta \|
  &\le (K+4) \sqrt{\delta n} + K \delta n \| x \|/n
  \le 5 \sqrt{K} \sqrt{K \delta n} \, .}
We now also require $n_8 \ge n_2(\p^1, \eps = 1/2, L = 5 \sqrt{K})$, where
$n_2$ is the constant in Lemma \ref{lem:beta-LCLT}. Then Lemma \ref{lem:beta-LCLT}(ii)
implies
\eqnsplst
{ \prob \Big( \Phi(V_n) = (x,n) \,\Big|\, \D_{(5)}', \D_{(6)}, \T_{n,m} \Big)
  &= \prob \Big( S_\beta(n - K \delta n) = x - z \Big) \\
  &\ge \frac{c(K)}{D_\beta^d (n - K \delta n)^{d/2}}
  \ge \frac{c'(K)}{D_\beta^d n^{d/2}} \, .}
On the other hand,
\eqnst
{ \prob \Big( \Phi(V_n) = (x,n) \,\Big|\, \T_{n,m} \Big)
  \le \frac{2 C}{D_\beta^d n^{d/2}} \, .}
The last two inequalities imply \eqref{e:srwestimate}.
\end{proof}

\subsection{Abundant intersections}
\label{ssec:enough}
We proceed with proving the second part of Theorem \ref{goodthingshappen}. To ease the presentation
of this calculation
let $\T_1$ and $\T_2$ be independent random trees distributed
as $\T_{\delta n, 2 \delta n}$ and rooted at $\rho_1, \rho_2$, respectively.
Let $\Phi_1$ and $\Phi_2$ be independent random walk mappings of
$\T_1$ and $\T_2$, respectively, into $\Z^d \times \Z_+$
such that $\Phi_1(\rho_1) = (o,0)$ and $\Phi_2(\rho_2) = (x,0)$.
Then on the event $\A(i) \cap \{ y_{i+K} - x'_{i+K} = x \}$,
the random variable $|\I|$ introduced in \eqref{e:I-defn}
has the same distribution as the random variable (also
denoted $|\I|$ here):
$$ |\I|
   = \sum_{U_1 \in \T_1, U_2 \in \T_2}
     {\bf 1}_{\text{$(U_1,U_2)$ intersect-well}} \, .$$
Here we have tacitly adapted the definition of ``intersect-well''
to the present setting, by replacing $\XX'_{i+K}$ by $\rho_1$ and
$\YY_{i+K}$ by $\rho_2$. Our goal in this section is to show that when
$d = 5$ we have $|\I| \geq c \sigma^4 D^{-5} (\delta n)^{1/2}$
with positive probability.

\begin{theorem}
\label{secondmomentargument}
Assume $d = 5$ and that $\| x \| \leq \sqrt{\delta n}$.
There exist constants $C < \infty$, $c > 0$ and
$n_9 = n_9 (\sigma^2, C_3, \p^1) < \infty$ such that
for $\delta n \ge n_9$ we have
\eqnst
{ \E |\I|
  \geq \frac{c \sigma^4}{D^5} \sqrt{\delta n} \, ,}
and
\eqn{e:2nd-moment-5D}
{ \E |\I|^2
  \leq \frac{C \sigma^8}{D^{10}} \delta n \, .}
\end{theorem}
%

Recall that for a tree vertex $V$ we write $h(V)$ for its distance from the root.
Also recall the vertices $Z_1$, $Z_1^+$, $Z_2$, $Z_2^+$ introduced
before Definition \ref{intersect}, and the constant $n_2$ of Lemma \ref{lem:beta-LCLT}(ii).

\begin{lemma}
\label{cltforfirstmoment}
Given instances of $\T_1$ and $\T_2$, let $U_1 \in \T_1$ and $U_2 \in \T_2$ be vertices both at height
$(5/6) \delta n \le h(U_1) = h(U_2) \leq \delta n$, and such that
$(1/2) \delta n \le h(Z_1), h(Z_2) \le (4/6) \delta n$.
There exists $c = c(d) > 0$ such that whenever
$\delta n \ge 6 n_2 (\p^1, \eps = 1/2, L = 1)$
and $\| x \| \le \sqrt{\delta n}$ we have
$$ \prob \big( \text{$(U_1, U_2)$ intersect-well} \, \big | \, \T_1, \T_2 \big )
   \geq \frac{c}{D^d (\delta n)^{d/2}} \, . $$
\end{lemma}
\begin{proof}
Denote the spatial locations of $Z_1, Z_1^+, Z_2, Z_2^+$ by
$z_1, z_1^+, z_2, z_2^+$, and denote the common spatial location
of $U_1$ and $U_2$ be $u$. Let us choose the spatial locations
so that the inequalities
\begin{align*}
  \| z_1 \| &\le \sqrt{(1/2) \delta n} & \| z_2 - x \| &\le \sqrt{(1/2) \delta n}    \\
  \| z_1^+ - z_1 \| &\le \sqrt{2} & \| z_2^+ - z_2 \| &\le \sqrt{2} \\
  \| z_1^+ - u \| &\le \sqrt{(1/6) \delta n} & \| z_2^+ - u \| &\le \sqrt{(1/6) \delta n}
\end{align*}
are satisfied --- this guarantees that the required events $\TS( \cdot, \cdot )$
all occur. Fix the displacements $z_1^+ - z_1$ and $z_2^+ - z_2$.
Since $\sqrt{1/2} + \sqrt{1/6} > 1/2$, there are
$\ge c D^{3d} (\delta n)^{3d/2}$ choices
for $(z_1,u,z_2)$ satisfying the requirements above.
Due to Lemma \ref{lem:beta-LCLT}, each choice has probability at least
$c (D^{-d} (\delta n)^{-d/2})^4$ occurring. Combined with
\eqref{e:simple-lb} to handle the displacements $z_1^+ - z_1$
and $z_2^+ - z_2$, this proves the statement of the lemma.
\end{proof}

\begin{lemma}
\label{lem:1st-mom-lb}
We have
\eqn{e:1st-mom-lb}
{ \E |\I|
  \geq \frac{c \sigma^4}{D^d} (\delta n)^{(6-d)/2} \, .}
whenever $\delta n \ge \max \{ n_2(\p^1, \eps = 1/2, L = 1), n_5, 6 n_6 \}$.
\end{lemma}

\begin{proof}
By Lemma \ref{cltforfirstmoment} we have
\eqn{e:first-moment-sum}
{ \E |\I|
  \ge \frac{c}{D^d (\delta n)^{d/2}} \sum_{h=5\delta n/6}^{\delta n} \,\,\, \sum_{k_1,k_2=\delta n/2}^{4\delta n/6}  \E \LL_{h,k_1} \E \LL_{h,k_2} }
where $\LL(h, k)$ counts the number of $U_1\in \T_1$ at level $h$ such that $Z_1$ is at level $k$.
 %
Note that since $Z_1$ is a backbone vertex, we have that
$$ \E \LL_{h,k_1} = \E \left[ \widetilde{\NN}_{h - k_1} \,|\, \widetilde{\NN}_{2 \delta n - k_1} = 0 \right] \, .$$
We have that
\begin{eqnarray*}
\E \left[ \widetilde{\NN}_{h - k_1} \,|\, \widetilde{\NN}_{2 \delta n - k_1} = 0 \right] &=& (1-\theta(2\delta n - k_1))^{-1} g_{h-k_1}'(1-\theta(2\delta n -h)) \\
&\geq& c \sigma^2 \, ,
\end{eqnarray*}
by Lemma \ref{survival} and Lemma \ref{lem:g_n'}. Summing this estimate in \eqref{e:first-moment-sum} concludes the proof.
%
\end{proof}


%

The remainder of this section is devoted to the proof of the
second moment estimate in Theorem \ref{secondmomentargument}.
Given numbers $h_u,h_w,k_1$ satisfying
$$ \delta n / 2 \leq k_1 \leq h_u, h_w \leq \delta n \, ,$$
we write $\LL(h_u, h_w, k_1)$ for the variable counting the number of pairs of tree vertices $U, W$ such that their highest common ancestor in the tree is at level $k_1$.

\begin{lemma}
\label{lem:treecount}
We have
\eqnsplst
{ &\E_{\T_{\delta n, 2\delta n}} \LL(h_u, h_w, k_1) \\
  &\qquad \leq (C_3 + 2\sigma^4 \delta n) {\bf 1}_{\{h_u > k_1, h_w > k_1\}}
      + (1 + 2\sigma^2 \delta n ){\bf 1}_{\{h_u = k_1 \text{ or } h_w = k_1\}} \, .}
\end{lemma}
\begin{proof} Let $\T_{\delta n, \infty}$ be a random tree obtained similarly to $\T_{\delta n, 2\delta n}$ dropping the requirement that the critical trees hanging on the backbone are conditioned not to reach level $2\delta n$. By the FKG inequality
\cite{Harris,FKG} we have
$$ \E_{\T_{\delta n, 2\delta n}} \LL(h_u, h_w, k_1) \leq \E_{\T_{\delta n, \infty}} \LL(h_u, h_w, k_1) \, ,$$
indeed, the measure $\T_{\delta n, 2\delta n}$ is obtained from $\T_{\delta n, \infty}$ by conditioning on a monotone decreasing event in a product measure (all the independent progeny random variables) and the random variable $\LL$ is monotone increasing. From here we will always calculate with respect to $\T_{\delta n, \infty}$ and we drop the corresponding subscript.

For two vertices $U,W$ at heights $h_u, h_w$ we write $S$ for their highest common ancestor at height $k_1$. There is a slight difference in the calculation depending on whether $S$ is in the backbone of $\T_{\delta n, \infty}$ or not. Write $\LL^1(h_u,h_w,k_1)$ for the number of such $U,W$ such that $S$ is not on the backbone and $\LL^2(h_u,h_w,k_1)$ when $S$ is on the backbone. We first estimate $\E \LL^1$. When $h_u > k_1$ and $h_w > k_1$, the expected number of pairs $U,W$ emanating from a fixed $S$ at height $k_1$ is at most
$$ \sum_{k = 2}^\infty \p(k) \, k \, (k - 1) \E \NN_{h_u - k_1 - 1} \, \E \NN_{h_w - \ell_1 - 1} = \sigma^2 \, .$$
When either $h_u=k_1$ or $h_w=k_1$ (that is, either $U$ or $W$ equal $S$) the expected number of such pairs is at most $1$. By summing over the backbone vertex from which $S$ emanates we have that
$$ \E \LL^1(h_u,h_w,k_1) \leq \sigma^4 \delta n  {\bf 1}_{\{h_u > k_1, h_w > k_1\}}  +  \sigma^2 \delta n  {\bf 1}_{\{h_u = k_1 \text{ or } h_w = k_1\}} \, .$$
To estimate $\E \LL^2$ we assume now that $S$ is the unique vertex on the backbone at height $k_1$, and when $h_u>k_1$ and $h_w>k_1$ the expected number of $U,W$ in $\T_{\delta n, \infty}(k_1)$ is
$$ \sum_{k = 2}^\infty \tilde{\p}(k) \, k \, (k - 1) \E \NN_{h_u - k_1 - 1} \, \E \NN_{h_w - \ell_1 - 1} \leq C_3 \, .$$
The expected number of $U,W$ such that $U \in \T_{\delta n, \infty}(k_1)$ but $W$ emanates from some other backbone vertex at height $> k_1$ is at most $\sigma^4 \delta n$. Similarly, the expected number of $U,W$ in which $h_u=k_1$ (and so $U=S$) is at most $\sigma^2\delta n$. Putting these together gives
\eqnsplst
{ &\E \LL^1(h_u,h_w,k_1) \\
  &\qquad \leq (C_3 + \sigma^4 \delta n) {\bf 1}_{\{h_u > k_1, h_w > k_1\}} +  (1+\sigma^2 \delta n ){\bf 1}_{\{h_u = k_1 \text{ or } h_w = k_1\}} \, .}
\end{proof}

Recall the constant $n_1(\p^1)$ of Lemma \ref{lem:beta-LCLT}
and the constant $L_1(\p^1)$ of \eqref{e:Green}.

\begin{lemma}
\label{lem:f-bnd}
Suppose $d \ge 3$. There are constants $C = C(d) > 0$
and $C_2 = C_2(\p^1)$ such that
\eqnst
{ \sum_{h : k_1 \vee k_2 \le h \le \delta n} \p^{2h - k_1 - k_2}(z_1,z_2)
  \le \frac{C}{D^d} f(k_1,k_2,z_1,z_2) \, ,}
where
\eqnst
{ f(k_1,k_2,z_1,z_2)
  := \begin{cases}
     |k_1 - k_2|^{(2-d)/2}
        & \parbox{4.5cm}{if $\| z_1 - z_2 \| \le |k_1 - k_2|^{1/2}$
        and $|k_1 - k_2| \ge n_1$;} \\
     & \\
     C_2
        & \text{if $\| z_1 - z_2 \| \le |k_1 - k_2|^{1/2} < n_1$;} \\
     & \\
    \| z_1 - z_2 \|^{2-d}
        & \parbox{4.5cm}{if $\| z_1 - z_2 \| > |k_1 - k_2|^{1/2}$
        and $\| z_1 - z_2 \| \ge L_1$;} \\
     & \\
     C_2
        & \text{if $|k_1 - k_2|^{1/2} < \| z_1 - z_2 \| < L_1$.}
     \end{cases} }
\end{lemma}

\begin{proof}
Suppose first we are in the case $\| z_1 - z_2 \| \le |k_1 - k_2|^{1/2}$.
Then for all $h \ge k_1 \vee k_2$ we have
$2h - k_1 - k_2 \ge |k_1 - k_2| \ge \|z_1 - z_2\|^2$. Hence
due to Lemma \ref{lem:beta-LCLT}, in the case when $|k_1 - k_2|$ is
large enough, we have
\eqnsplst
{ \sum_{h : k_1 \vee k_2 \le h \le \delta n} \p^{2h - k_1 - k_2}(z_1,z_2)
  &\le \sum_{k = |k_1 - k_2|}^\infty \p^k(z_1,z_2) \\
  &\le \frac{C}{D^d} \sum_{k = |k_1 - k_2|}^\infty k^{-d/2} \\
  &\le \frac{C}{D^d} |k_1 - k_2|^{(2-d)/2} .}
When $|k_1 - k_2|$ is not large, the bound follows
trivially.

Suppose now we are in the other case $\| z_1 - z_2 \| > |k_1 - k_2|^{1/2}$.
Then due to \eqref{e:Green}, in the case when $\| z_1 - z_2 \|$ is large enough,
we have
\eqnst
{ \sum_{h : k_1 \vee k_2 \le h \le \delta n} \p^{2h - k_1 - k_2}(z_1,z_2)
  \le \sum_{k = 0}^\infty \p^k(z_1,z_2)
  \le \frac{C}{D^d} {\| z_1 - z_2 \|^{2-d}} \, .}
The bound is trivial in the case when $\| z_1 - z_2 \|$ is not large.
\end{proof}

\noindent {\bf Proof of Theorem \ref{secondmomentargument}.}
The lower bound on the first moment is Lemma \ref{lem:1st-mom-lb}
(we require that $n_9 \ge \max \{ 6 n_2(\p^1, \eps = 1/2, L = 1),\, n_5, 6 n_6 \}$).
We are left to prove the upper bound on the second moment.
First we drop the requirements of ``typically spaced'' from the definition of $\I$. This gives that
\be
\label{e:2nd-mom-expr}
\begin{split}
  \E |\I^2|
  &\leq \sum_{k_1, k_2 =\delta n/2}^{\delta n} \sum_{h_u, h_w = k_1 \vee k_2}^{\delta n}
        \E \LL(h_u,h_w,k_1) \E \LL(h_u,h_w,k_2) \\
  &\qquad\qquad\qquad\qquad\qquad \times p(h_u,h_w,k_1,k_2) \, ,
\end{split}
\ee
where $p(h_u,h_w,k_1,k_2)$ is the probability that $\Phi_1(U_1)=\Phi_2(U_2)$ and $\Phi_1(W_1) = \Phi_2(W_2)$ where $U_1, U_2, W_1, W_2$ are any tree vertices satisfying that the highest common ancestor of $U_1$ and $W_1$ is at height $k_1$ and the highest common ancestor of $U_2$ and $W_2$ is at height $k_2$ and $h(U_1) = h(U_2) = h_u$ and $h(W_1)=h(W_2)= h_w$. Note that this probability only depends on the corresponding heights and not on the vertices.
We have that
\eqnspl{e:diagram}
{ p(h_u,h_w,k_1,k_2) = \sum_{z_1, z_2 \in \Z^d} \sum_{u, w \in \Z^d}
    &\p^{k_1}(o,z_1) \p^{h_u - k_1}(z_1,u) \p^{h_w - k_1}(z_1,w) \\
    &\times \p^{k_2}(x,z_2) \p^{h_u - k_2}(z_2,u) \p^{h_w - k_2}(z_2,w) \, .}
We can perform the summations over $u, w$ yielding the expression
\eqnspl{e:diagram2}
{ \sum_{z_1, z_2 \in \Z^d}
    &\p^{k_1}(o,z_1) \p^{2 h_u - k_1 - k_2}(z_1,z_2)
    \p^{2 h_w - k_1 - k_2}(z_1,z_2) \p^{k_2}(x,z_2)  \, .}

Using Lemmas \ref{lem:treecount} and \ref{lem:f-bnd} we sum \eqref{e:2nd-mom-expr} over $h_u, h_w > k_1 \vee k_2$ and 
we get a bound of
$$ Y_1 = {C (\sigma^4 \delta n)^2 \over D^{2d}} \sum_{k_1, k_2 =\delta n/2} ^{\delta n} \sum_{z_1,z_2} f(k_1,k_2,z_1,z_2)^2 \p^{k_1}(o,z_1) \p^{k_2}(x,z_2) , .$$
Similarly, we sum over $h_u>k_1 \vee k_2$ and $h_w = k_1 \vee k_2$ and when the roles of $h_u$ and $h_w$ exchanged, getting a bound of
\be Y_2 = {C(\sigma^3 \delta n)^2 \over D^d} \sum_{k_1, k_2 =\delta n/2} ^{\delta n} \sum_{z_1,z_2} f(k_1,k_2,z_1,z_2)  \p^{|k_1-k_2|} (z_1,z_2) \p^{k_1}(o,z_1) \p^{k_2}(x,z_2) \, .\nonumber \ee
And finally our third bound is when $h_u=h_w=k_1 \vee k_2$ giving
$$ Y_3 = C (\sigma^2 \delta n )^2 \sum_{k_1, k_2 =\delta n/2} ^{\delta n} \sum_{z_1,z_2}  \p^{k_1}(o,z_1) \p^{|k_1 - k_2|}(z_1,z_2)
    \p^{|k_1 - k_2|}(z_1,z_2) \p^{k_2}(x,z_2) \, ,$$
so that $\E |\I|^2 \leq Y_1 + Y_2 + Y_3$. We start with bounding $Y_1$.
We split the summation over $z_1, z_2$
into two parts:
\begin{itemize}
\item[(I)] $\| z_2 - z_1 \| \le |k_1 - k_2|^{1/2}$;
\item[(II)] $\| z_2 - z_1 \| > |k_1 - k_2|^{1/2}$.
\end{itemize}
For the bounds we are going to require $n_9 \ge 2 n_1$.
We first bound case (I), and initially restrict to
$|k_1 - k_2| \ge n_1$ where $n_1$ is from Lemma \ref{lem:f-bnd}.
Using Lemma \ref{lem:beta-LCLT} and
$k_1 \ge \delta n /2 \ge n_9/2 \ge n_1$ in the first step,
the sum over $z_1, z_2$ in $Y_1$ is at most
\eqnspl{e:short-Case(I)}
{ &\sum_{z_2 \in \Z^d} \ \sum_{z_1 : \| z_2 - z_1 \| \le |k_1 - k_2|^{1/2}}
     \p^{k_1}(0,z_1) \p^{k_2}(x,z_2)
     |k_1 - k_2|^{2-d} \\
  &\qquad \le \frac{C}{D^d (\delta n)^{d/2}}
     \sum_{z_2 \in \Z^d} \ \sum_{z_1 : \| z_2 - z_1 \| \le |k_1 - k_2|^{1/2}}
     \p^{k_2}(x,z_2)
     |k_1 - k_2|^{2-d} \\
  &\qquad \le \frac{C}{(\delta n)^{d/2}}
     \sum_{z_2 \in \Z^d}
     \p^{k_2}(x,z_2)
     |k_1 - k_2|^{2-\frac{d}{2}} = \frac{C |k_1 - k_2|^{2-\frac{d}{2}}}{(\delta n)^{d/2}} \, .}
Now we sum this over $k_1, k_2$ and get a bound of $C (\delta n)^{4-d}$. Similarly, when summing over $k_1, k_2$ satisfying $|k_1-k_2|\leq n_1$ we get a bound of $C(\delta n)^{1-d/2}$ which is negligible since $d<6$. Putting all these together gives a contribution to $Y_1$ from case $1$ that is of order $D^{-2d} \sigma^8 (\delta n)^{6-d}$.

In case (II) we initially restrict to $\| z_1 - z_2 \| \ge L_1$.
We have
\eqnsplst
{ &\sum_{z_2 \in \Z^d} \ \sum_{z_1 : \| z_1 - z_2 \| > |k_1 - k_2|^{1/2}}
     \p^{k_1}(0,z_1) \p^{k_2}(x,z_2)
     \| z_1 - z_2 \|^{4-2d} \\
  &\qquad \le \frac{C}{D^d (\delta n)^{d/2}}
    \sum_{z_2 \in \Z^d} \ \sum_{z_1 : \| z_1 - z_2 \| > |k_1 - k_2|^{1/2}}
     \p^{k_2}(x,z_2)
     \| z_1 - z_2 \|^{4-2d} \\
  &\qquad \le
    \frac{C}{(\delta n)^{d/2}}
    \sum_{z_2 \in \Z^d} \p^{k_2}(x,z_2)
    |k_1 - k_2|^{(4-d)/2} = \frac{C |k_1 - k_2|^{2-\frac{d}{2}}}{(\delta n)^{d/2}} \, .}
The case $\| z_1 - z_2 \| \leq L_1$ is dealt with similarly, and all together we get that
$n_9$ can be chosen in such a way that
$$ Y_1 \leq CD^{-2d} \sigma^8 (\delta n)^{6-d} \, .$$
Very similar calculations yield that
$$ Y_2 \leq C D^{-2d} \sigma^6 (\delta n)^{3-d/2} \, , \qquad Y_3 \leq C D^{-2d} \sigma^4 (\delta n)^{3-d/2} \, ,$$
concluding the proof. \qed\\

\noindent {\bf Proof of Theorem \ref{goodthingshappen}.}
The first part of the theorem is just a combination of Lemmas \ref{lem:l1-exists}, \ref{lem:Yi's exist}, \ref{lem:l2-exists}, \ref{lem:X_{i+K} etc} and \ref{lem:show-K-spatial-good}, where we take
\eqnst
{ n_3
  = \max \{ n_5, n_5', 6 n_6 (3, \sigma^2, C_3), n_7, n_8(\p^1,K), n_9(\sigma^2, C_3, \p^1) \}. }
For the second part of the theorem we now choose $c_0 = c/2$, where $c$
is the constant in the lower bound on the first moment in
Theorem \ref{secondmomentargument}. Then the second statement of
Theorem \ref{goodthingshappen} follows immediately from
Theorem \ref{secondmomentargument} together with the inequality
$$ \prob \Big( V \geq \frac{1}{2} \E V \Big)
   \geq { (\E V)^2 \over 4 \E V^2 } \, ,$$
valid for any non-negative random variable $V$.
\qed

\section{Analysis of tree bad blocks}
\label{sec-badtree}

In this section we bound the resistance between $\Phi(X_i)$ and $\Phi(X_{i+K})$ conditioned on one of the good events in Definition \ref{ktreegood} not occurring. We will give a bound in terms of the following quantity, which later we will bound inductively. For any $k\leq n$ define
\be\label{l:gammabar} \bar{\gamma}(k; (x,n))
   = \sum_{y \in \Z^d} {\p^k(o,y) \p^{n-k}(y,x) \over \p^n(o,x)} \gamma(k,y) \, .\ee

For $a = 1, \ldots, 6$ we define $\cH_{\mathrm{(a)}}$ to be the event that conditions
(1) to ($a-1$) in Definitions \ref{ktreegood} and \ref{kspatialgood} are satisfied,
but condition ($a$) is not. Then we may write the disjoint union
$$ \A(i)^c \cup \B(i,c_0)^c
   = \bigcup_{a=1}^6 \cH_{\mathrm{(a)}} \cup \left( \A(i) \cap B(i,c_0)^c \right) \, .$$
Recall the constants $n_5$, $n_7$ introduced in \eqref{e:Kolm2} Lemma \ref{survival}.


\begin{lemma}
\label{badtree1}
There exist $C_4 > 0$ and $\delta_2 > 0$ such that
\eqnsplst
{ &\E \Big [ \Reff(\Phi(X_i) \lra \Phi(X_{i+K})) \, \mid \, \cH_{\mathrm{(1)}}, \Phi(V_n) = (x,n) \Big ] \\
  &\qquad \leq (1 + C_4 \delta) \bar{\gamma}(\delta n; (x,n))
     + (K-1) \bar{\gamma}(\delta n; (x,n)) }
whenever $0 < \delta < \delta_2$, $\delta n \ge \max \{ n_5(C_3), n_7(C_3) \}$.
\end{lemma}

\begin{proof} By the triangle inequality of effective resistance \eqref{triangle} we have
\eqnspl{e:R-triang}
{ &\E \Big [ \Reff(\Phi(X_i) \lra \Phi(X_{i+K})) \, \Big| \, \cH_{(1)}, \Phi(V_n) = (x,n) \Big ] \\
  &\qquad \leq \sum_{i'=i}^{i+K-1} \E \Big [ \Reff(\Phi(X_{i'}) \lra \Phi(X_{i'+1})) \, \Big| \,
      \cH_{(1)}, \Phi(V_n) = (x,n) \Big ] \, .}
The terms $i' = i+1, \ldots, i+K-1$ are not affected by the conditioning on $\cH_{(1)}$, and hence
we get the $(K-1) \bar{\gamma}(\delta n; (x,n))$ term. So it remains to prove that
\eqnsplst
{ &\E \Big [ \Reff(\Phi(X_i) \lra \Phi(X_{i+1})) {\bf 1}_{\cH_{(1)}} \, \mid \, \Phi(V_n) = (x,n) \Big ] \\
  &\qquad\qquad \leq (1+O(\delta)) \bar{\gamma}(\delta n; (x,n)) \prob( \cH_{(1)} | \Phi(V_n) = (x,n) ) \, .}

If $\cH_{(1)}$ occurs, then precisely one of the following three disjoint events must happen:
\begin{enumerate}
\item[(i)] There are no levels in $[i\delta n, (i+1)\delta n)$ that reach height $(i+2)\delta n$,
\item[(ii)] There are more than one such levels,
\item[(iii)] There is a unique such level $\ell_1$ but $\ell_1 \not \in [(i+1/4) \delta n, (i+3/4)\delta n]$.
\end{enumerate}
We handle each of these separately. If (i) occurs, then the trees emanating from each level are conditioned
not to reach level $(i+2)\delta n$. Hence,
\eqnsplst
{ &\E \Big [ \Reff(\Phi(X_i) \lra \Phi(X_{i+1})) {\bf 1}_{\hbox{(i)}} \, \mid \, \Phi(V_n) = (x,n) \Big ] \\
  &\qquad\qquad \leq \bar{\gamma} (\delta n; (x,n)) \prob( \hbox{(i)} \, | \, \Phi(V_n) = (x,n)) \, ,}
since in the definition of $\gamma(\delta n; (x,n))$ we take a supremum over $m \geq 2 \delta n$. \\

In handling the event (ii), the following notation will be convenient.
We write $\Reff(\Phi(X_i) \stackrel{\G}{\lra} \Phi(X_{i+1}))$ for the
effective resistance evaluated in a given graph $\G$.
If (ii) occurs, then let $j_1, \ldots, j_k$ be the set of levels in $[i\delta n, (i+1)\delta n)$ such that
$k \geq 2$ and $\T_{n,m}(j_s)$ reaches level $(i+2)\delta n$ but not level $m$ for all
$s = 1, \ldots, k$ and denote by $\F(j_1,\ldots, j_k)$ this event. Let $\T_{n,\infty}$ be defined as $\T_{n,m}$ only without the conditioning on the side branches.
%
We have
\eqnsplst
{ &\E \Big[ \Reff \Big( \Phi(X_i) \stackrel{\T_{n,m}}{\lra} \Phi(X_{i+1}) \Big) \, \Big| \,
     \hbox{(ii)},\, \Phi(V_n) = (x,n) \Big] \\
  &\qquad = \sum_{\substack{k \ge 2 \\ (j_1,\dots,j_k)}}
     \E \Big[ \Reff \Big( \Phi(X_i) \stackrel{\T_{n,\infty}}{\lra} \Phi(X_{i+1}) \Big) \, \Big| \,
     \F(j_1,\dots,j_k),\, \Phi(V_n) = (x,n) \Big] \\
  &\qquad\qquad\qquad \times \prob ( \F(j_1,\dots,j_k) \,|\, \hbox{(ii)} ) \, ,}
since the events in question require that all side branches emanating from $V_{i \delta n}$ to $V_{(i+K)\delta n}$ do not reach level $m$. During the rest of the proof of (ii) we work where $\T_{n,\infty}$ is the background measure.

Write $\F'(j_1,\ldots j_k)$ for the same event as $F(j_1,\dots,j_k)$ except that the trees
$\T_{n,\infty}(j_s)$ are now only required to reach level $(i+2)\delta n$
(and may perhaps reach level $m$ as well). Since $\F \subseteq \F'$ we have
\eqnspl{e:F-by-F'}
{ &\E \Big [ \Reff(\Phi(X_i) {\lra}
      \Phi(X_{i+1})) {\bf 1}_{\F(j_1,\ldots, j_k)} \, \mid \, \Phi(V_n) = (x,n) \Big ] \\
  &\qquad \leq \E \Big [ \Reff(\Phi(X_i) {\lra}
      \Phi(X_{i+1})) {\bf 1}_{\F'(j_1,\ldots, j_k)} \, \mid \,
     \Phi(V_n) = (x,n) \Big ] \, .}
Since $\F'$ is an increasing event and
$\Reff(\Phi(X_i) \lra \Phi(X_{i+1}))$ is a decreasing random variable,
the FKG inequality \cite{Harris,FKG} implies that the right hand side of
\eqref{e:F-by-F'} is at most
\eqnsplst
{ &\E \Big [ \Reff(\Phi(X_i) \stackrel{\T_{n,\infty}}{\lra}
     \Phi(X_{i+1}))  \, \mid \, \Phi(V_n) = (x,n) \Big ]
     \prob ( \F'(j_1,\ldots, j_k) ) \\
  &\qquad \le \prob ( \F'(j_1,\ldots, j_k) ) \, \bar{\gamma}(\delta n; (x,n)) \, ,}
where in the last step we are using that $\T_{n,\infty}$ is the weak limit as
$m \to \infty$ of $\T_{n,m}$.

We need to bound the ratio between the probability of $\F$ and $\F'$. Write $\NN$ for the total number of
progeny at level $(i+2) \delta n$ of $\T_{n,\infty}(j_1), \ldots, \T_{n,\infty}(j_k)$. Then,
\eqnsplst
{ \prob(\F)
  &\geq \prob(\F') \E \big[ (1 - \theta(m - (i+2)\delta n))^\NN \, \big| \, \F' \big] \\
  &\geq \prob(\F') \E \Big[ \left(1 - 6 (\sigma^2 n)^{-1} \right)^\NN \, \Big| \, \F' \Big] \, ,}
where the last inequality is by $m - (i+2) \delta n \ge (1 - 2 \delta) n \ge n/2 \ge n_5$ and
our estimate on $\theta$ \eqref{e:Kolm2}. Note that
$\NN = \NN^{(1)} + \ldots + \NN^{(k)}$ where $\NN^{(1)}, \ldots, \NN^{(k)}$ are independent and
$\NN^{(s)}$ has the distribution of $\widetilde{\NN}_{(i+2) \delta n - j_s}$,
$s = 1, \dots, k$. Hence,
\eqnsplst
{ \E \Big[ \left( 1 - 6 (\sigma^2 n)^{-1} \right)^\NN \Big| \F' \Big]
  &\geq \prod_{s=1}^k \E \Big[ \left( 1 - 6 (\sigma^2 n)^{-1} \right)^{\NN^{(s)}} \,
      \Big| \, \NN^{(s)} > 0 \big] \\
  &\geq \prod_{s=1}^k \E \big[ ( 1 - \NN^{(s)} 6 (\sigma^2 n)^{-1} ) \, \mid \,
      \NN^{(s)} > 0 \big] \\
  &\geq (1 - O(\delta))^k \, ,}
since $\E[ \NN^{(s)} | \NN^{(s)} > 0] = \E \widetilde{\NN}_{2 \delta n - j_s} \prob ( \widetilde{\NN}_{2 \delta n - j_s} > 0 )^{-1}
= O( \sigma^2 \delta n )$ when $\delta n \ge n_7$, by \eqref{e:survival-uncond}. Hence,
$$ \prob(\F'(j_1,\ldots, j_k)) \leq (1 + O(\delta))^k \prob(\F(j_1,\ldots,j_k)) \, .$$
Therefore,
\eqnsplst
{ &\E \Big [ \Reff(\Phi(X_i) \stackrel{\T_{n,m}}{\lra}
      \Phi(X_{i+1}))  \, \mid \, \F(j_1,\ldots, j_k), \Phi(V_n) = (x,n) \Big ] \\
  &\qquad\qquad \leq (1+O(\delta))^k \bar{\gamma}(\delta n; (x,n)) \, .}

In the tree $\T_{n,\infty}$, and hence in the tree $\T_{n,m}$,
the number of vertices $V_k$ on the backbone
that reach $(i+2)\delta n$ is stochastically bounded above by a Binomial random
variable with parameters $\delta n$ and $p = {C \over \delta n}$,
by \eqref{e:survival-uncond}. Hence, the probability
that there are precisely $k$ such vertices is at most $e^{-ck}$ for some $c > 0$.
We get that as long as $\delta > 0$ is small enough (as a function of $c$) we have
\eqnsplst
{ &\E \Big [ \Reff(\Phi(X_i) \stackrel{\T_{n,m}}{\lra}
      \Phi(X_{i+1})) {\bf 1}_{\hbox{(ii)}} \, \mid \, \Phi(V_n) = (x,n) \Big ] \\
  &\qquad\qquad \leq (1+O(\delta)) \bar{\gamma} (\delta n; (x,n))
      \prob( \hbox{(ii)} \, | \, \Phi(V_n) = (x,n)) \, ,}
concluding the analysis of (ii). \\

If (iii) occurs, then there is a unique $\ell_1$ which reaches level $(i+2)\delta n$ but not
$m$ and all other levels do not reach level $(i+2)\delta n$. A similar analysis as in (ii) with $k=1$
using the FKG inequality gives that
\eqnsplst
{ &\E \Big [ \Reff(\Phi(X_i) \lra \Phi(X_{i+1})) {\bf 1}_{\hbox{(iii)}} \, \mid \, \Phi(V_n) = (x,n) \Big ] \\
  &\qquad\qquad \leq (1+O(\delta)) \bar{\gamma} (\delta n; (x,n)) \prob( \hbox{(iii)} \, | \, \Phi(V_n) = (x,n)) \, .}
\end{proof}

\begin{lemma}
\label{badtree2}
\begin{equation*}
\begin{split}
  &\E \Big [ \Reff(\Phi(X_i) \lra \Phi(X_{i+K})) \, \Big| \,
      \cH_{\mathrm{(2)}},\, \ell_1,\, \Phi(V_n) = (x,n) \Big ] \\
  &\qquad \leq \bar{\gamma}(\ell_1 - i\delta n; (x,n)) + 1
      + \bar{\gamma}((i+1)\delta n - \ell_1-1) \\
  &\qquad\quad + (K-1) \bar{\gamma}(\delta n; (x,n)) \, .
\end{split}
\end{equation*}
\end{lemma}

\begin{proof}
As in the previous lemma, we use the triangle inequality as in \eqref{e:R-triang}
with $\cH_{\mathrm{(1)}}$ now replaced by $\cH_{\mathrm{(2)}}$. Again, the terms
containing $\Reff(\Phi(X_{i'}) \lra \Phi(X_{i'+1}))$ for $i' = i + 1, \dots, i + K - 1$
are unaffected by the conditioning, and hence contribute the term $(K-1) \bar{\gamma}(\delta n; (x,n))$.
%
The rest of the lemma is much easier than the previous one, since on the event that
Definition \ref{ktreegood}(1) is satisfied, the backbone $V_{i\delta n},\ldots, V_{\ell_1}$ together with its side branches (not counting the side branch of $V_{\ell_1}$) is distributed as $\T_{\ell_1 - i \delta n, 2 \delta n}$, and the backbone $V_{\ell_1+1},\ldots V_{(i+1)\delta n}$ together with its side branches (again, not counting the side branch of $V_{(i+1)\delta n}$) is distributed as $\T_{(i+1) \delta n - \ell_1 - 1, 2 \delta n - \ell_1 - 1}$.
%
Hence we get
\eqnsplst
{ &\E \Big [ \Reff(\Phi(X_{i}) \lra \Phi(X_{i+1})) \, \Big| \,
      \cH_{\mathrm{(2)}},\, \ell_1,\, \Phi(V_n) = (x,n) \Big ] \\
  &\qquad \leq \bar{\gamma}(\ell_1 - i \delta n; (x,n)) + 1 + \bar{\gamma}((i+1)\delta n - \ell_1-1) \, ,}
as required.
\end{proof}

\begin{lemma} 
\label{badtree3}
We have
\begin{equation*}
\begin{split}
  &\E \Big [ \Reff(\Phi(X_i) \lra \Phi(X_{i+K})) \, \Big| \,
     \cH_{\mathrm{(3)}},\, \ell_1,\, \Phi(V_n) = (x,n) \Big ] \\
  &\qquad \leq \bar{\gamma}(\ell_1 - i \delta n; (x,n)) + 1
     + \bar{\gamma}((i+1) \delta n - \ell_1-1) \\
  &\qquad\quad + (K-2) \bar{\gamma}(\delta n; (x,n)) + (1+C_4\delta) \bar{\gamma}(\delta n; (x,n))
\end{split}
\end{equation*}
whenever $0 < \delta < \delta_2$, $\delta n \ge \max \{ n_5(C_3), n_7(C_3) \}$.
\end{lemma}

\begin{proof}
We again start with the triangle inequality as in \eqref{e:R-triang}, with
$\cH_{\mathrm{(1)}}$ now replaced by $\cH_{\mathrm{(3)}}$.
An argument almost identical to that of Lemma \ref{badtree1}, yields that
the term involving $\Reff(\Phi(X_{i+K-1}) \lra \Phi(X_{i+K}))$
is bounded by $(1 + C_4 \delta) \bar{\gamma}(\delta n; (x,n))$. The
rest of the terms are bounded as in Lemma \ref{badtree2}.
\end{proof}

\begin{lemma}
\label{badtree4}
\begin{equation*}
\begin{split}
  &\E \Big [ \Reff(\Phi(X_i) \lra \Phi(X_{i+K})) \, \Big| \,
      \cH_{\mathrm{(4)}},\, \ell_1,\, \ell_2,\, \Phi(V_n) = (x,n) \Big ] \\
  &\qquad \leq \bar{\gamma}(\ell_1 - i\delta n; (x,n)) + 1
      + \bar{\gamma}((i+1)\delta n - \ell_1-1) \\
  &\qquad\quad + (K-2) \bar{\gamma}(\delta n; (x,n))
      +  \bar{\gamma}(\ell_2 - ( (i + K - 1) \delta n; (x,n) ) + 1 \\
  &\qquad\quad + \bar{\gamma}( (i + K) \delta n - \ell_2 - 1; (x,n) ) \, .
\end{split}
\end{equation*}
\end{lemma}

\begin{proof}
Similarly to the proof of Lemma \ref{badtree2}, we bound the resistance using
subgraphs that conditioned on Definition \ref{ktreegood}(1),(2),(3) holding
(and conditioned on the values of $\ell_1$, $\ell_2$) are independent of
whether (4) holds or not.
\end{proof}

\section{Analysis of spatially bad blocks}
\label{sec-badspatial}

In this section we analyze what happens when condition (5) or (6) in
Definition \ref{kspatialgood} fails, that is, some spatial
displacement is ``not typical'', and also what happens when $\B(i,c_0)$ fails.
Let us introduce some notation.
We write $\Gtree$ for the event
$$ \Gtree = \{ \hbox{(1)--(4)}, \ell_1, \ell_2, \Phi(V_n) = (x,n) \} \, .$$
We define a set of times
$i \delta n = T_0 < T_1 < \dots < T_{K+4} = (i+K) \delta n$,
time differences $t_1, t_2, \ldots, t_{K+4}$ and
spatial locations $z_0, \ldots , z_{K+4} \in \Z^d$ by
\begin{align*}
z_0     &= x_i            &         &                              & T_0     &= i\delta n\\
z_1     &= v_{\ell_1}     & t_1     &= \ell_1 - i \delta n         & T_1     &= \ell_1 \\
z_2     &= v_{\ell_1 + 1} & t_2     &= 1                           & T_2     &= \ell_1 + 1 \\
z_3     &= x_{i+1}        & t_3     &= (i+1) \delta n - \ell_1 - 1 & T_3     &= (i+1) \delta n \\
z_4     &= x_{i+2}        & t_4     &= \delta n                    & T_4     &= (i+2) \delta n \\
z_5     &= x_{i+3}        & t_5     &= \delta n                    & T_5     &= (i+3) \delta n \\
        &\ \ \vdots       &         &\ \ \vdots                    &         &\ \ \vdots \\
z_{K+1} &= x_{i+K-1}      & t_{K+1} &= \delta n                    & T_{K+1} &= (i+K-1) \delta n \\
z_{K+2} &= v_{\ell_2}     & t_{K+2} &= \ell_2 - (i+K-1) \delta n   & T_{K+2} &= \ell_2 \\
z_{K+3} &= v_{\ell_2+1}   & t_{K+3} &= 1                           & T_{K+3} &= \ell_2 + 1 \\
z_{K+4} &= x_{i+K}        & t_{K+4} &= K \delta n - \ell_2 - 1     & T_{K+4} &= (i+K) \delta n
\end{align*}
Observe that conditional on $\Gtree$, the times $T_s$ and time differences
$t_s$ are non-random but the spatial locations $z_s$ are random.
Furthermore, we define for any $s = 1,\ldots, K+4$
\eqn{e:qs}
{ \q_s(z)
  = \sum_{\substack{\| y_r \| \leq \sqrt{t_r} \\ r = 1, \dots, s-1 \\
    y_1 + \cdots + y_{s-1} = z}}
    \prod_{r=1}^{s-1} \p^{t_r}(0,y_r) \, .}
Finally, for any $s = 1, \ldots, K+4$ we define the event $\EE_{(5)}^s$ by
\eqnst
{ \EE_{(5)}^s
  = \bigcap_{r=1}^{s-1} \Big \{ \| z_r - z_{r-1} \| \leq {\sqrt{t_r}} \Big \} \ \bigcap \
    \Big \{ \| z_s - z_{s-1} \| > {\sqrt{t_s}} \Big \} \, .}
Note that
\eqn{e:pes}
{ \prob(\EE_{(5)}^s \, | \, \Gtree)
  = \sum_{z, y : \| y \| > \sqrt{t_s}} \q_s(z)
    {\p^{t_s}(z,z+y) \p^{n-(T_s-T_0)}(z+y,x) \over \p^n(o,x)} \, .}

\begin{lemma}
\label{spatialbad}
For any $s = 1, \ldots, K+4$ and $s' = 1, \ldots, K+4$ the quantity
$$ \cR_{s',s}
   = \E \Big[ \Reff( (z_{s'-1}, T_{s'-1}) \lra (z_{s'}, T_{s'}) )
     {\bf 1}_{\EE_{(5)}^s} \, \Big| \, \Gtree \big] $$
satisfies:
$$ \cR_{s',s}
   \le \prob(\EE_{(5)}^s \, | \, \Gtree) \, , \quad \text{when $s' = 2,\, K+3$,} $$
$$ \cR_{s',s}
   \le \prob(\EE_{(5)}^s \, | \, \Gtree) \gamma(t_{s'}) \, , \quad
       \text{when $s' < s$, $s' \not= 2,\, K+3$,} $$
$$ \cR_{s,s'}
   \le \sum_{z} \q_s(z) \sum_{y : \| y \| > \sqrt{t_s}}
       {\p^{t_s}(o,y) \p^{n-T_s+T_0}(y,x-z) \over \p^n(o,x)} \gamma(t_s,y) \, ,$$
when $s' = s$, $s' \not= 2,\, K+3$,
$$ \cR_{s',s}
   \le \sum_{\substack{z, y_s, y : \\ \|y_s\|>\sqrt{t_s}}}
       \q_s(z) {\p^{t_s}(o,y_s) \p^{t_{s'}}(o,y) \p^{n-t_{s'}-T_s}(y, x-z-y_s) \over \p^n(o,x)}
       \gamma(t_{s'},y) \, ,$$
when $s' > s$, $s' \not= 2,\, K+3$.
     %
\end{lemma}

\begin{proof}
The case $s' = 2,\, K+3$ is trivial, so assume $s' \not= 2,\, K+3$.
The case $s' < s$ is easy. Condition on $\EE_{(5)}^s$ and on the spatial locations
$z_0, z_1, \ldots, z_s$ such that $\EE_{(5)}^s$ holds. Since $\| z_{s'} - z_{s'-1} \| \leq \sqrt{t_{s'}}$
we may bound the resistance between the corresponding points by $\gamma(t_{s'})$.

In order to handle the case $s' = s$, we condition on $z_0, z_1, \ldots, z_s$ such that
$\EE_{(5)}^s$ holds. With this conditioning the required resistance is bounded above by $\gamma(t_{s}, z_{s}-z_{s-1})$. So the required expectation is bounded above by
\eqnst
{ \sum_{\substack{z_0, z_1, \ldots, z_s : \\ \| z_r - z_{r-1} \| \leq \sqrt{t_r} \\
     r = 1, \dots, s-1 \\
     \| z_{s} - z_{s-1} \| > \sqrt{t_s}}}
     { \p^{T_0}(o,z_0) \over \p^{n}(o,x)} \Big[ \prod_{r=1}^s \p^{t_r}(z_{r-1},z_r) \Big]
     \p^{n - T_s}(z_s,x) \gamma(t_{s}, z_{s} - z_{s-1}) \, .}
By changing variables $y_1 = z_1-z_0, y_2=z_2-z_1, \ldots, y_{s-1} = z_{s-1}-z_{s-2}$ and $y= z_s-z_{s-1}$ and $z=y_1+\cdots+y_{s-1}$ this equals
\eqnst
{ \sideset{}{'}\sum_{z_0, y_1, \ldots, y_{s-1}, y} { \p^{T_0}(o,z_0) \over \p^{n}(o,x)}
     \Big[ \prod_{r=1}^{s-1} \p^{t_r}(o,y_r) \Big] \p^{t_s}(o,y)
     \p^{n-T_s}(z_0+z+y,x) \gamma(t_s,y) \, ,}
where $\sideset{}{'}\sum$ indicates the restriction
$\| y_1 \| \leq \sqrt{t_1}, \ldots, \| y_{s-1} \| \leq \sqrt{t_s}, \| y \| > \sqrt{t_s}$.
By summing over $z_0$ this simplifies to
\eqnsplst
{ &\sum_{z} \q_s(z) \sum_{y : \| y \| > \sqrt{t_s}}
    {\p^{t_s}(o,y) \p^{n-T_s+T_0}(y,x-z) \over \p^n(o,x)} \gamma(t_s,y) }
as required.

The case $s' > s$ is done similarly. The required expectation is bounded above by
\eqnsplst
{ \sum_{\substack{z_0, z_1, \ldots, z_s : \\ \| z_r - z_{r-1} \| \leq \sqrt{t_r} \\
     r = 1, \dots, s-1 \\
     \| z_{s} - z_{s-1} \| > \sqrt{t_s}}} \sum_{z_{s'-1},y \in \Z^d} { \p^{T_0}(o,z_0) \over \p^{n}(o,x)} \Big[ \prod_{r=1}^s \p^{t_r}(z_{r-1},z_r) \Big] \p^{T_{s'-1}-T_s}(z_{s},z_{s'-1}) \\ \times \p^{t_{s'}}(z_{s'-1},z_{s'-1}+y)\p^{n-T_{s'}}(z_{s'-1}+y,x) \gamma(t_{s'},y) \, .
          }
Summing over $z_0, z_{s'-1}$ and recalling \eqref{e:qs} simplifies this to
\eqnsplst {  \sum_{\substack{z, y_s, y : \\ \|y_s\|>\sqrt{t_s}}} \q_s(z) {\p^{t_s}(o,y_s) \p^{t_{s'}}(o,y) \p^{n-t_{s'}-T_s}(y, x-z-y_s) \over \p^n(o,x)} \gamma(t_{s'},y) }
\end{proof}

Next, to handle part (6) of Definition \ref{kspatialgood} recall that we defined
\eqnst
{ \EE_{\mathrm{(6)}}
  = \bigcap_{s=1}^{K+4} \big \{ \| z_s - z_{s-1} \| \leq \sqrt{t_s} \big \}\
    \bigcap \ \big \{\text{one of the conditions in (6) fails}\big \} \, .}
We also define $\EE_{\mathrm{(7)}} = \A(i) \cap \B(i,c_0)^c$.

\begin{lemma}
\label{badtree6}
\eqnsplst
{ &\E \Big[ \Reff(\Phi(X_i) \lra \Phi(X_{i+K}))  \, \Big| \,
      {\bf 1}_{\EE_{\mathrm{(6)}} \cup \EE_{\mathrm{(7)}}}, \Gtree \Big] \\
  &\qquad \leq (K - 2) \gamma(\delta n) + \gamma(\ell_1 - i \delta n)
      + 1 + \gamma((i+1) \delta n - \ell_1 - 1) \\
  &\qquad\quad + \gamma(\ell_2 - (i+K-1) \delta n) + 1 + \gamma((i+K) \delta n - \ell_2 - 1) \, .}
\end{lemma}

\begin{proof}
Condition on $\Gtree$ and $\EE_{\mathrm{(6)}} \cup \EE_{\mathrm{(7)}}$. We have that
$\| z_s - z_{s-1} \| \leq \sqrt{t_s}$ for all $s = 1, \ldots, K+4$.
Hence, under this conditioning, we may bound the resistance between
$(z_{s-1}, T_{s-1})$ and $(z_s,T_s)$ by $\gamma(t_s)$, concluding the proof.
\end{proof}

We close this section with a bound on the resistance on
the ``final stretch'' between $X_{i^{\mathrm{last}}}$ and $V_n$, where
$i^{\mathrm{last}} = K ( \lf (K \delta)^{-1} \rf - 1 )$.
Observe that $K \delta n \le n - i^{\mathrm{last}} \delta n < 2 K \delta n$,
and write
\eqnst
{ n - i^{\mathrm{last}} \delta n
  = K' \delta n + \tilde{n} \, ,}
where $K \le K' \le 2K - 2$ and $\delta n \le \tilde{n} < 2 \delta n$.

\begin{lemma}
\label{lem:final-stretch}
\eqnsplst
{ &\E \Big[ \Reff(\Phi(X_{i^{\mathrm{last}}}) \lra \Phi(V_n)) \,\Big|\,
      \Phi(V_n) = (x,n) \Big] \\
  &\qquad \le K' \bar{\gamma}(\delta n; (x,n)) + \bar{\gamma}(\tilde{n}; (x,n)) \, .}
\end{lemma}

\begin{proof}
This follows from the triangle inequality for resistance.
\end{proof}

\section{Analysis of good blocks}
\label{sec:goodblocks}

In this section we will estimate expectations of resistances given the event
$$ \Ggood = \big \{ \Phi(V_n)=(x,n), \A(i), \B(i,c_0), \ell_1, \ell_2 \big \} \, .$$

\begin{lemma}
\label{easyestimates}
Conditional on $\Ggood$, we have
\begin{enumerate}
\item $\E \big [ \Reff( \Phi(X_i) \lra \Phi(V_{\ell_1}) ) \, \mid \, \Ggood \big ] \leq \gamma(\ell_1 - i \delta n)$.
\item $\E \big [ \Reff( \Phi(V_{\ell_1}) \lra \Phi(X_{i+1}) ) \, \mid \, \Ggood \big ] \leq \gamma((i+1) \delta n - \ell_1-1)+1$.
\item For all $i+1 \leq j \leq i+K-2$ we have
$$ \E \big [ \Reff( \Phi(X_j) \lra \Phi(X_{j+1}) ) \, \mid \, \Ggood \big ] \leq \gamma(\delta n) \, .$$
\item $\E \big [ \Reff( \Phi(X_{i+K-1}) \lra \Phi(V_{\ell_2}) ) \, \mid \, \Ggood \big ] \leq \gamma(\ell_2 - (i+K-1)\delta n)$.
\item $\E \big [ \Reff( \Phi(V_{\ell_2}) \lra \Phi(X_{i+K}) )  \, \mid \, \Ggood \big ] \leq \gamma((i+K)\delta n - \ell_2 -1) + 1$.
\item $\E \big [ \Reff( \Phi(V_{\ell_2}) \lra \Phi(X'_{i+K})) \, \mid \, \Ggood \big ] \leq \gamma((i+K)\delta n - \ell_2 -1) + 1$.
\item $\E \big [ \Reff( \Phi(V_{\ell_1}) \lra \Phi(Y_{i+1}) ) \, \mid \, \Ggood \big ] \leq \gamma((i+1) \delta n - \ell_1-1)+1$.
\item For all $i+1 \leq j \leq i+K-1$ we have
$$ \E \big [ \Reff( \Phi(Y_j) \lra \Phi(Y_{j+1}) ) \, \mid \, \Ggood \big ] \leq \gamma(\delta n) \, .$$
\end{enumerate}
\end{lemma}
\begin{proof}
The proof of (i), (iii), (iv) and (viii) is immediate by Definition \ref{kspatialgood}.
The other estimates follow almost as quickly by Definition \ref{kspatialgood} and triangle inequality
for resistance.
\end{proof}


Recall the constant $n_1$ from Lemma \ref{lem:beta-LCLT}(i)
and the constant $n_9$ from Theorem \ref{secondmomentargument}.

\begin{lemma}
\label{hardestimate}
Assume $d = 5$. There exists $C_5 < \infty$ such that we have
\eqnst
{ \E \big [ \Reff( \Phi(X'_{i+K}) \lra \Phi(Y_{i+K}) \, \mid \, \Ggood \big ]
  \leq C_5 \max_{1 \le \ell \le \delta n} \gamma(\ell) }
whenever $\delta n \ge \max \{ n_1(\p^1), n_9(\sigma^2, C_3, \p^1) \}$.
\end{lemma}

For convenience we will prove Lemma \ref{hardestimate} under the assumption that there exists an $M$
such that the progeny distribution is bounded by $M$ with probability $1$. Then by taking $M\to \infty$
and keeping $n$ fixed we obtain Lemma \ref{hardestimate} in our usual generality. This is possible,
since $C_5$ does not depend on $M$, and the restriction on $\delta n$ only depends on $\sigma^2, C_3$,
so it is sufficient to approximate $\{ p(k) \}$ by some $\{ p_M(k) \}$ in such a way
that
\eqnsplst
{ 1
  &= \sum_{k = 0}^M p_M(k) = \sum_{k = 0}^M k p_M(k) \, , \\
  \sigma^2
  &= \lim_{M \to \infty} \sum_{k \ge 1} k(k-1) p_M(k) \, ,\\
  C_3
  &\ge \sup_{M \ge 1} \sum_{k \ge 1} k^3 p_M(k) \, . }
Therefore in the rest of this section we assume the bound $M$.

Given any $n$ and $m$ such that $m\geq 2n$ we regard the random tree $\T_{n,m}$ as a subtree of an infinite $M$-ary tree $T_M$ with root $\rho$ as follows: the root of $\T_{n,m}$ is mapped to $\rho$ and if $W$ is a vertex of $\T_{n,m}$ with $k$ children we map the $k$ edges randomly amongst the ${M \choose k}$ possible choices in $T_M$. Denote by $\V_n \in T_M$ the random vertex where the last backbone vertex of $\T_{n,m}$ was mapped to. The triple $(\T_{n,m}, \rho, \V_n)$ is a doubly rooted tree.
%
%
%
%
%
Define
\eqnst
{ q(k)
  := \binom{M}{k}^{-1} p(k) \, .}

\begin{lemma}\label{proboftree} For a fixed triple $(t,\rho,V)$ where $t \subset T_M$ is a tree and $V\in T_M$ at height $n$ such that $t$ does not reach level $m$ and $V$ has no children in $t$, we have
\eqn{e:T1-formula}
{ \prob ( (\T_{n,m},\rho, \V_n) = (t, \rho, V) )
  \propto
    \prod_{\substack{W \in t \\ W \not= V}} q(\deg^+_t(W)) \, ,}
where $\deg^+_t(W)$ is the number of children of $W$ in $t$.
\end{lemma}
\begin{proof}
Let $\rho=V_0, V_1, \ldots, V_n=V$ be the unique path in $t$ from $\rho$ to $V$. The probability that $\T_{n,m}=t$ with this backbone equals
$$ {1 \over Z} \prod_{i=0}^{n-1} \tilde{p}(\deg^+_t(V_i)-1) \cdot \prod_{W \in t\setminus\{V_0,\ldots, V_n\}} p(\deg^+_t(W)) \, ,$$
where $Z= \prod_{i=0}^{n-1} \theta(m-i)$. Hence, the probability that $(\T_{n,m},\rho, \V_n) = (t, \rho, V)$ (as subtrees of $T_M$) equals
\eqnsplst{ {1 \over Z} & \prod_{i=0}^{n-1} \Bigg [ M {M-1\choose \deg^+_t(V_i)-1}\Bigg]^{-1} \tilde{p}(\deg^+_t(V_i)-1)  \\ \times &\prod_{W \in t\setminus\{V_0,\ldots, V_n\}}\Bigg[ {M\choose \deg^+_t(W)}\Bigg]^{-1} p(\deg^+_t(W)) \, .
}
Manipulating with $\tilde{p}(k-1)=kp(k)$ finishes the proof.
\end{proof}

%

%
%
%

For the statement of the next lemma we fix
\eqnst
{ 0 \le k_1 \leq \delta n-1 \qquad\qquad
  k_1 + 1 \leq h_u \leq \delta n 
}
Given $V \in T_M$ at level
$\delta n$ and $U \in T_M$ at level $h_u$ let $Z\in T_M$ be the highest common ancestor of $V$ and $U$ and let $Z^+$ be the unique child of $Z$ leading towards $U$. Given a tree $t \subset T_M$ such that $V, U\in t$ and $V$ does not have any children in $t$, we have a unique decomposition of $t$ into edge disjoint trees $(t^A, \rho, Z), (t^B, Z^+, U), t^C$ and $t^D$, see figure \ref{fig:to-U1}. The doubly rooted tree $(t^A, \rho, Z)$ contains all the descendants of $\rho$ that are not descendants of $Z$. The doubly rooted tree $(t^B, Z^+, U)$ contains all the descendants of $Z^+$ that are not descendants of $U$. The tree $t^C$ contains all the descendants of $U$ and finally the tree $t^D$ contains all other edges, namely, all the descendants of $Z$ that are not descendants of $Z^+$ (in particular, the edge $Z,Z^+$ is in $t^D$).

For $W \in T_M$ let $\Theta_W$ denote the tree isomorphism that takes $W$ to $\rho$ and the descendants subtree of $W$ onto $T_M$.

%
%

\psfrag{deltan}{$\delta n$}
\psfrag{2deltan}{$2 \delta n$}
\psfrag{ell1}{$k_1$}
\psfrag{h_u}{$h_u$}
\psfrag{zero}{$0$}
\psfrag{Y_1}{$V$}
\psfrag{U_1}{$\!U$}
\psfrag{Z_1}{$\!\!\!Z$}
\psfrag{Z^+_1}{$Z^+$}
\psfrag{rho}{$\rho$}
\psfrag{t1}{$t^A$}
\psfrag{t2}{$t^B$}
\psfrag{t3}{$t^C$}
\psfrag{t4}{$t^D$}

\begin{figure}
\includegraphics[scale=0.6]{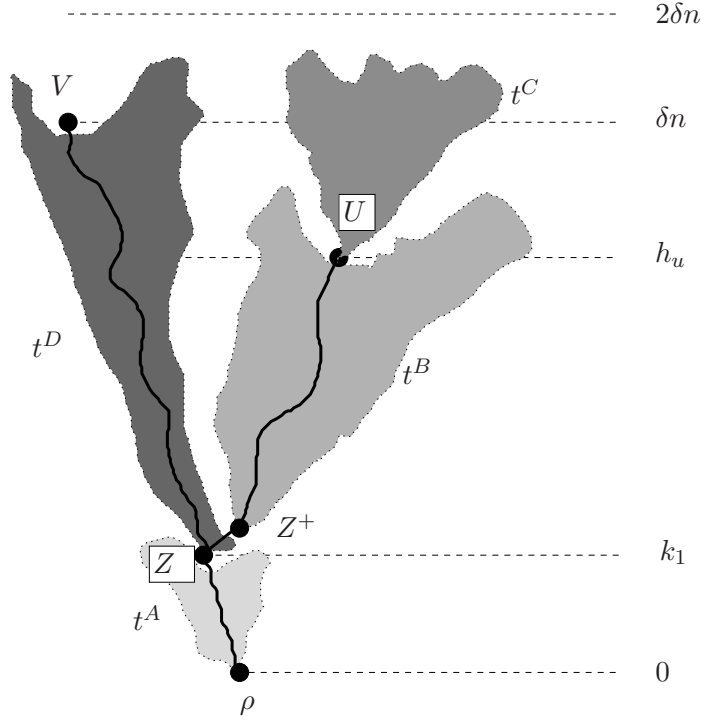}
\caption{Illustration of the decomposition into
edge-disjoint trees $t^A, t^B, t^C, t^D$ appearing
in Lemma \ref{lem:main-new} ($2 \delta n$ and $\delta n$
are not to scale).}
\label{fig:to-U1}
\end{figure}

\begin{lemma}
\label{lem:main-new}
Let $V, U\in T_M$ be at heights $\delta n$ and $h_u$, respectively and $(\T,\rho, \V)$ be distributed as
$(\T_{\delta n, 2\delta n},\rho, \V_{\delta n})$. Conditionally on the event $\{\V = V \, , U \in \T\}$
we have that
\eqnst
{ (\T^A, \rho, Z)
  \stackrel{\mathrm{d}}{=} (\T_{k_1, 2\delta n}, \rho, \V_{k_1}) \,\mid\, \V_{k_1}=Z \, ,}
and
\eqnst
{ \Theta_{Z^+}((\T^B, Z^+, U))
  \stackrel{\mathrm{d}}{=} (\T_{h_u-k_1-1, 2\delta n - k_1-1}, \rho, \V_{h_u-k_1-1}) \,\mid\,
     \V_{h_u-k_1-1}=\Theta_{Z^+}(U) \, .}
%
%
\end{lemma}
%

\begin{proof} For any $t\subset T_M$ that contains $V$ and $U$ (and $V$ has no children in $t$) by lemma \ref{proboftree} we have
$$ \prob((\T, \rho, \V) = (t, \rho, V)) \propto \prod_{\substack{W \in t \\ W \not= V}} q(\deg^+_t(W)) \, .$$
We factorize the right hand side so it equals
$$ \prod_{\substack{W \in t^A \\ W \not= Z}} q(\deg^+_t(W)) \cdot \prod_{\substack{W \in t^B \\ W \not= U}} q(\deg^+_t(W)) \cdot \prod_{\substack{W \in t^C}} q(\deg^+_t(W)) \cdot \prod_{\substack{W \in t^D \\ W \not= V}} q(\deg^+_t(W)) \, .$$
By summing over all the possible values of $t^B, t^C$ and $t^D$ we get that
$$ \prob((\T^A, \rho, Z) = (t^A, \rho, Z)) \propto \prod_{\substack{W \in t^A \\ W \not= Z}} q(\deg^+_t(W)) \, ,$$
which gives the claim for $\T^A$ by Lemma \ref{proboftree}. The same argument works similarly for $\T^B$ noting that under the shift $\Theta_{Z^+}$ the degrees do not change.
\end{proof}

\noindent{\bf Proof of Lemma \ref{hardestimate}.}
All our expectations in the following proof are conditioned on the event $\A(i), \Phi(V_n)=(x,n)$.

Let $(\T_1, \rho, \V_1)$ and $(\T_2, \rho, \V_2)$ be two independent copies of $(\T_{\delta n, 2\delta n}, \rho, \V_{\delta n})$ randomly embedded into $T_M$ as before. Conditionally on $\A(i)$ let $\Phi_1$ and $\Phi_2$ be independent random walk
mappings of $T_M$ such that $\Phi_1(\rho)=\Phi(\XX'_{i+K})$ and $\Phi_2(\rho)=\Phi(\YY_{i+K})$, so that $\|\Phi_1(\rho) - \Phi_2(\rho)\|\leq \sqrt{\delta n}$. In this way, the required quantity
$\Reff( \Phi(\XX'_{i+K}) \lra \Phi(\YY_{i+K}))$ is distributed as $\Reff( \Phi_1(\rho) \lra \Phi_2(\rho))$ where in the latter, the resistance is computed in the graph $\Phi_1(\T_1) \cup \Phi_2(\T_2)$.

For notational convenience, and without loss of generality, we assume
that $\Phi_1(\rho_1) = (o,0)$, $\Phi_2(\rho_2) = (x,0)$,
with $\| x \| \le \sqrt{\delta n}$.
Recall the notation $Z_1, Z_1^+, Z_2, Z_2^+$ introduced after
Definition \ref{def:kgood}. Definition \ref{intersect}
adapted to the current setting reads as follows:

\begin{definition}
\label{intersect-random}
We say that the vertices $U_1, U_2 \in T_M$ \emph{intersect-well}
if:
\begin{itemize}
\item[1.] $U_1 \in \T_1$ and $U_2 \in \T_2$;
\item[2.] $(5/6) \delta n \le h(U_1) = h(U_2) \le \delta n$;
\item[3.] $(1/2) \delta n \le h(Z_1), h(Z_2) \le (4/6) \delta n$;
\item[4.] $\TS(\rho_1, Z_1)$, $\TS(Z_1^+, U_1)$, $\TS(\rho_2, Z_2)$, $\TS(Z_2^+, U_2)$;
\item[5.] $\Phi_1(U_1) = \Phi_2(U_2)$.
\end{itemize}
Define $\tilde{\I}$ by
\begin{equation}
\label{e:I-defn-random}
\begin{split}
  \tilde{\I}
  = \left\{ (U_1,U_2) : U_1, U_2 \text{intersect-well} \right\}\, .
\end{split}
\end{equation}

\end{definition}
Then it is clear that $|\tilde{\I}|$ has the same distribution as
$|\I|$ introduced earlier.
Recall that $\B=\B(i,c_0)$ is the event
$\{ |\I| \geq c_0 \frac{\sigma^4}{D^d} (\delta n)^{(6-d)/2} \}$.
Conditional on $\T_1, \T_2, \Phi_1, \Phi_2$, and the event
$\B$, draw a pair $(\U_1,\U_2)$ from the set $\tilde{\I}$, uniformly at random.
This is possible, since on the event $\B$ we have $|\tilde{\I}| > 0$. Denote
$$ \RR
   = \Reff ( \Phi_1(\rho_1) \lra \Phi_2(\rho_2) ) \, .$$
Writing for short
\eqnsplst
{ {\bf 1}_{V}
  &= {\bf 1}_{\V_1=V_1,\, \V_2 = V_2} \, ,\\
  {\bf 1}_{U}
  &= {\bf 1}_{\U_1 = U_1,\, \U_2 = U_2} \, , }
we have
\eqnspl{e:sum1}
{ &\E [ \RR {\bf 1}_\B ]
  = \sum_{V_1, V_2 \in T_M} \
    \sum_{U_1, U_2 \in T_M}
    \E [ \RR {\bf 1}_\B {\bf 1}_{V} {\bf 1}_{U} ] \, .}
Recall that $V_1$ and $U_1$ determine the vertices
$Z_1, Z_1^+$, and $V_2$ and $U_2$ determine
$Z_2, Z_2^+$. The first sum in \eqref{e:sum1} is over
all pairs $V_1, V_2 \in T_M$ at height $\delta n$. The second sum is over all pairs $U_1,U_2 \in T_M$
such that
\eqnsplst
{ (5/6) \delta n
  \le h_u
  &:= h(U_1) = h(U_2) \le \delta n \, ,\\
  \delta n / 2
  \le k_1
  &:= h(Z_1) \le (4/6) \delta n \, ,\\
  \delta n / 2
  \le k_2
  &:= h(Z_2) \le (4/6) \delta n \, .}
Given $z_1, z_1^+, z_2, z_2^+, u \in \Z^d$, write ${\bf 1}_{\Phi}$ for short
for the indicator function of the intersection of the following six events:
\begin{align*}
  \Phi_1(Z_1)
  &= (z_1, k_1) &
  \Phi_2(Z_2)
  &= (z_2, k_2) \\
  \Phi_1(Z_1^+)
  &= (z_1^+, k_1 + 1) &
  \Phi_2(Z_2^+)
  &= (z_2^+, k_2 + 1) \\
  \Phi_1(U_1)
  &= (u, h_u) &
  \Phi_2(U_2) &= (u, h_u)
\end{align*}
This allows us to rewrite \eqref{e:sum1} in the form:
\eqnspl{e:sum2}
{ &\E [ \RR {\bf 1}_\B ]
  = \sum_{V_1, V_2} \
    \sum_{U_1, U_2} \ \
    \sideset{}{'}\sum_{\substack{u, z_1, z_1^+ \\ z_2, z_2^+}} \ \
    \E [ \RR {\bf 1}_\B {\bf 1}_{V}
    {\bf 1}_{U} {\bf 1}_{\Phi} ] \, .}
Here the prime on the summation over $z_1, z_1^+, z_2, z_2^+, u$
indicates that these vertices are restricted to choices
that are compatible with the occurrence of
$\TS(\rho_1,Z_1)$, $\TS(Z_1^+,U_1)$, $\TS(\rho_2,Z_2)$, $\TS(Z_2^+,U_2)$,
that is, $\| z_1 \| \le \sqrt{k_1}$, $\| u - z_1^+ \| \le \sqrt{h_u - k_1 - 1}$,
$\| z_2 - x \| \le \sqrt{k_2}$, $\| u - z_2^+ \| \le \sqrt{h_u - k_2 - 1}$.

In the presence of the indicators on the right hand side
of \eqref{e:sum2} we can also insert the indicator
\eqnsplst
{ {\bf 1}_{T}
  &= {\bf 1}_{U_1 \in \T_1,\, U_2 \in \T_2} \, ,}
as this event already occurs.
Hence the expectation on the right hand side of
\eqref{e:sum2} equals
\eqn{e:in-T12}
{ \E [ \RR {\bf 1}_\B {\bf 1}_{V}
     {\bf 1}_{U} {\bf 1}_{T} {\bf 1}_{\Phi} ] \, .}
Observe that we have
\eqn{e:cond-I}
{ \E [ {\bf 1}_{U} \,|\, \B,\, (\T_1, \V_1),
    (\T_2, \V_2), \Phi_1, \Phi_2 ]
  = \frac{1}{|\tilde{\I}|}
  \le \frac{D^d}{c_0 \sigma^4 (\delta n)^{(6-d)/2}} \, ,}
and that $\RR$ and the other indicators in \eqref{e:in-T12}
are measurable with respect to the conditioning
in \eqref{e:cond-I}. Hence
\eqnspl{e:use-B}
{ \E [ \RR {\bf 1}_\B {\bf 1}_{V}
     {\bf 1}_{U} {\bf 1}_{T} {\bf 1}_{\Phi} ]
  &\le \frac{D^d}{c_0 \sigma^4 (\delta n)^{(6-d)/2}}
     \E [ \RR {\bf 1}_\B {\bf 1}_{V}
     {\bf 1}_{T} {\bf 1}_{\Phi} ] \\
  &\le \frac{D^d}{c_0 \sigma^4 (\delta n)^{(6-d)/2}}
     \E [ \RR {\bf 1}_{V} {\bf 1}_{T}
     {\bf 1}_{\Phi} ]\, .}
In order to bound $\RR$ from above, we define
\eqnst
{ \RR_1
  = \Reff (\Phi_1(\rho_1) \lra \Phi_1(U_1))
  \qquad \text{ and } \qquad
  \RR_2
  = \Reff (\Phi_2(\rho_2) \lra \Phi_2(U_2)) \, ,}
and by the triangle inequality for effective resistance \eqref{triangle} we have
$$ \RR \leq \RR_1 + \RR_2 \, ,$$
on the event ${\bf 1}_{T} {\bf 1}_{\Phi}$. Inserting this into \eqref{e:in-T12} yields
\eqnsplst
{ \E [ \RR {\bf 1}_\B ]
  &\le \frac{D^d}{c_0 \sigma^4 (\delta n)^{(6-d)/2}} \sum_{\substack{V_1, V_2,\\ U_1, U_2}}
   \sideset{}{'}\sum_{\substack{u, z_1, z_1^+ \\ z_2, z_2^+}}
  \Big(
    \E [ \RR_1 {\bf 1}_{V} {\bf 1}_{T}
    {\bf 1}_{\Phi} ]
  + \E [ \RR_2 {\bf 1}_{V} {\bf 1}_{T}
    {\bf 1}_{\Phi} ]
    \Big) \, .}
We only analyze the term containing $\RR_1$, since the arguments for handling $\RR_2$ are identical.
We bound $\RR_1$ from above by,
\eqnsplst
{ \RR_1
  &\le \Reff ( \Phi_1(\rho_1) \lra \Phi_1(Z_1) )
      + 1 + \Reff ( \Phi_1(Z_1^+) \lra \Phi_1(U_1) ) \, .}
Due to Lemma \ref{lem:main-new}, conditioned on the events
in the indicators ${\bf 1}_{V} {\bf 1}_{T}$, the
distribution of $\T_1^A$ is the same
as that of $\T_{k_1,2 \delta n}$, and the distribution of
$\T_1^B$ is the same as the distribution
of $\T_{h_u - k_1 - 1, 2 \delta n - k_1 - 1}$.
Due to the presence of the indicator ${\bf 1}_{\Phi}$,
that fixes the spatial locations of
$\Phi_1(Z_1), \Phi_1(Z_1^+), \Phi_1(U_1)$ (respectively)
to be $z_1, z_1^+, u$ (respectively), we have
\eqnsplst
{ \E [ \Reff ( \Phi_1(\rho_1) \lra \Phi_1(Z_1) )
     {\bf 1}_{V} {\bf 1}_{T} {\bf 1}_{\Phi} ]
  &= \gamma(k_1, z_1)
    \E [ {\bf 1}_{V} {\bf 1}_{T} {\bf 1}_{\Phi} ] \\
  &\le \gamma(k_1)
     \E [ {\bf 1}_{V} {\bf 1}_{T} {\bf 1}_{\Phi} ] \, ,}
and
\eqnsplst
{ \E [ \Reff ( \Phi_1(Z_1^+) \lra \Phi_1(U_1) )
     {\bf 1}_{V} {\bf 1}_{T} {\bf 1}_{\Phi} ]
  &= \gamma(h_u - k_1 - 1, u - z_1^+)
    \E [ {\bf 1}_{V} {\bf 1}_{T} {\bf 1}_{\Phi} ] \\
  &\le \gamma(h_u - k_1 - 1)
     \E [ {\bf 1}_{V} {\bf 1}_{T} {\bf 1}_{\Phi} ] \, .}
Together with analogous bounds for $\RR_2$, this yields
\eqnspl{e:sum3}
{ \E [ \RR {\bf 1}_\B ]
  &\le \frac{D^d}{c_0 \sigma^4 (\delta n)^{(6-d)/2}} \sum_{\substack{V_1, V_2,\\ U_1, U_2}} \ \
    \sideset{}{'}\sum_{\substack{u, z_1, z_1^+ \\ z_2, z_2^+}} \ \
    \E [ {\bf 1}_{V} {\bf 1}_{T} {\bf 1}_{\Phi} ] \\
  &\qquad \times \left( \gamma(k_1) + \gamma(h_u - k_1 - 1)
    + \gamma(k_2) + \gamma(h_u - k_2 - 1) +  2 \right) \\
  &\le \frac{ (4\max_{0 \le k \le \delta n} \gamma(k)+2) D^d}{c_0 \sigma^4 (\delta n)^{(6-d)/2}}
    \sum_{\substack{V_1, V_2,\\ U_1, U_2}}  \ \
    \sideset{}{'}\sum_{\substack{u, z_1, z_1^+ \\ z_2, z_2^+}} \ \
    \E [ {\bf 1}_{V} {\bf 1}_{T} {\bf 1}_{\Phi} ] \, .}
We have
\eqnsplst
{ \E [ {\bf 1}_{V} {\bf 1}_{T} {\bf 1}_{\Phi} ]
  &= \E [ {\bf 1}_{V} {\bf 1}_{T} ]
    \p^{k_1}(o,z_1) \p^1(z_1,z_1^+)
    \p^{h_u - k_1 - 1}(z_1^+, u) \\
  &\qquad \times \p^{k_2}(x, z_2) \p^1(z_2, z_2^+)
    \p^{h_u - k_2 - 1}(z_2^+, u) \, . }
Removing the restrictions involved in the primed summation
in \eqref{e:sum3} we can perform the convolutions
of the transition probabilities and get that
\eqnspl{e:sum4}
{ \E [ \RR {\bf 1}_\B ]
  &\le \frac{(4 \max_{1 \le k \le \delta n} \gamma(k) + 2) D^d}{c_0 \sigma^4 (\delta n)^{(6-d)/2}}
    \sum_{\substack{V_1, V_2,\\ U_1, U_2}}
    \E [ {\bf 1}_{V} {\bf 1}_{T} ] \p^{2 h_u} (o,x) \, .}
By the local central limit theorem, and due to $\| x \| \le \sqrt{\delta n}$,
$h_u \ge (5/6) \delta n \ge n_1/2$, we have
\eqn{e:LCLT}
{ \p^{2 h_u}(o,x)
  \le \frac{C}{D^d (\delta n)^{d/2}} \, .}

Now, fix $V_1, V_2$ and $h_u$ and sum ${\bf 1}_{T}$ on $U_1, U_2$. This number is bounded by the product of the number of vertices
of $\T_1$ and $\T_2$ at height $h_u$, respectively. Note that this
random variable is independent of ${\bf 1}_V$, and is a product of two independent variables
that have the same distribution, namely,  the number of vertices
of $\T_{\delta n, 2 \delta n}$ at level $h_u$.
The latter is stochastically smaller than the number of
vertices of $\T_{\delta n,\infty}$ at level $h_u$, which has expectation
$\sum_{k_1 < h_u} \E \widetilde{\NN}_{h_u - k_1} \le \sigma^2 \delta n$.
Finally, note that
\eqn{e:sumV}
{ \sum_{V_1, V_2} \E [ {\bf 1}_{V} ]
  = 1 \, .}
Putting together \eqref{e:sum4}, \eqref{e:LCLT}, and \eqref{e:sumV} we get:
\eqnst
{ \E [ \RR {\bf 1}_\B ]
  \le \frac{ (4 \max_{1 \le k \le \delta n} \gamma(k) + 2) D^d}{c_0 \sigma^4 (\delta n)^{(6-d)/2}}
      \sum_{h_u = (6/4) \delta n}^{\delta n} \frac{(\sigma^2 \delta n)^2}{D^d (\delta n)^{d/2}}
  \le C \max_{1 \le k \le \delta n} \gamma(k) \, .}
An appeal to the second part of Theorem \ref{goodthingshappen} concludes the proof.\qed \\

\noindent {\bf Proof of Theorem \ref{goodblock}.}
We choose $n_5 = \max \{ n_1(\p^1), n_9(\sigma^2, C_3, \p^1) \}$.
Note the elementary inequality
${1 \over {\gamma_1^{-1} + \gamma_2^{-1}}} \leq {\gamma_1 + \gamma_2 \over 4}$.
We apply this inequality to the resistances of the two graphs ``in parallel''
between $V_{\ell_1}$ and $V_{\ell_2}$: one via the backbone and one via the
vertices $\YY_{i+1}, \dots, \YY_{i+K}, \XX'_{i+K}$.
The parallel law \eqref{parallel} and Lemmas \ref{easyestimates} and \ref{hardestimate} gives
\eqnsplst
{ &\E \big[ \Reff (\Phi(X_i) \lra \Phi(X_{i+K}) \,\big|\, \Ggood \big] \\
  &\quad \leq \gamma(\ell_1 - i \delta n) + 1 + \gamma ((i+K)\delta n - \ell_2 - 1) \\
  &\quad\quad + \frac14 \left[ 1 + \gamma((i+1) \delta n - \ell_1 - 1)
    + (K-2) \gamma(\delta n) + \gamma(\ell_2 - (i+K-1) \delta n) \right] \\
  &\quad\quad + \frac14 \big[ 1 + \gamma((i+1) \delta n - \ell_1 - 1) + (K-1) \gamma(\delta n) \\
  &\quad\quad\qquad + 1 + \gamma ((i+K)\delta n - \ell_2 - 1) + C_5 \max_{1 \le k \le \delta n} \gamma(k) \big] \, .}
Choosing $K$ large with respect to $C_5$ concludes the proof of the theorem. \qed

\section{Proof of Theorem \ref{mainthm}}

Let $K_0$ be the constant in Theorem \ref{goodblock}. We fix
$K = K_0$ for the remainder of the proof. Let
\eqnst
{ n_0
  = \max \{ n_3(\sigma^2, C_3, \p^1, K), n_4(\sigma^2,C_3,\p^1), 4 k_1(\p^1) \} \, ,}
where $n_3$ and $n_4$ are the constants from Theorems \ref{goodthingshappen}
and \ref{goodblock} and $k_1$ is the constant from Proposition \ref{prop:rw-estimates}.
Let $\delta_0 > 0$, $\alpha \in (0,1/2)$ and $A > 0$ be constants. These
will be chosen below in the order: $\delta_0, \alpha, A$, and among others
we will require that
\eqnspl{e:K+3}
{ \delta_0 \le (K+4)^{-1} \, ,\qquad\qquad \delta_0 \le \delta_1 \, ,}
where $\delta_1$ is the constant from Proposition \ref{prop:rw-estimates}(ii).
Once $\delta_0$ and $\alpha$ will be chosen, we choose $A$ to satisfy:
\eqn{e:As}
{ A \ge n_0 / \delta_0 \, , \qquad\qquad
  A^{-1/(2 \alpha)} \le 1/\sqrt{n_0} \, .}
We prove the theorem by induction. Since $A \geq n_0 / \delta_0$,
the theorem holds for all $n < n_0 / \delta_0$,
so we may assume $n \ge n_0 / \delta_0$. Our induction hypothesis is that
for all $n' < n$ and all $x \in \Z^d$ we have
$$ \gamma(n',x)
   \leq A (n')^{1-\alpha} \Big ( { \| x \|^2 \over n' } \vee 1 \Big )^\alpha \, ,$$
and given the hypothesis we prove it for $n$. Since $\gamma(n,x) \leq n$
it suffices to prove when $\|x\| \leq n A^{-1/(2 \alpha)}$. Note that this
implies $\| x \| \le n/\sqrt{n_0}$.
Now, given such $x$ fix
\eqn{e:def-delta}
{ \delta
  = \min \Big\{ \eta \ge \min \{ \delta_0, n/\|x\|^2 \} :
         \text{$\eta n$ is an integer} \Big\} \, .}
Note that
\eqn{e:deltan-good}
{ \delta n
  \geq \min \left\{ \delta_0 n,\, \frac{n^2}{\| x \|^2} \right\}
  \ge n_0 \, ,}
and
\eqn{e:xnorm-good}
{ \|x\|
  \leq \sqrt{2n / \delta} \, ,}
so Theorem \ref{goodthingshappen} can be applied to $(x,n)$.

Consider the sequences
\eqnst
{ (0, \ldots, K), (K, \ldots, 2K), \ldots, ((N-1)K, \ldots, NK) \, ,}
where $N = \lf (\delta K)^{-1} \rf - 1$ is the number of sequences.
Fix any integer $m \geq 2n$ and define $\gamma_m(n,x)$ to be
$$ \gamma_m(n,x)
   = \E_{\T_{n,m}} \big[ \Reff \big( (o,0) \lra \Phi(V_n) \big) \, \big| \,
     \Phi(V_n) = (x,n) \big] \, ,$$
where we consider the resistance in the graph $\Phi(\T_{n,m})$,
so that $\gamma(n,x) = \sup_{m \geq 2n} \gamma_{m}(n,x)$.
We bound $\gamma_m(n,x)$ by estimating 
$$ \E \big[ \Reff(\Phi(X_i) \lra \Phi(X_{i+K}) \, \big| \,
       \Phi(V_n) = (x,n) \big] \, ,$$
for each $i = 0, K, 2K, \ldots, (N-1)K$ and then adding
these up using the triangle inequality for resistance \eqref{triangle},
also adding the estimate for the final stretch from
$NK \delta n$ to $n$.

Fix such an $i$.
We split the above expectation according to whether $\A(i) \cap \B(i,c_0)$ occurred.
By Theorem \ref{goodblock} we have that
\eqnspl{goodbound}
{ &\E \big[ \Reff(\Phi(X_i) \lra \Phi(X_{i+K})) {\bf 1}_{\A(i) \cap \B(i,c_0)} \, \big| \,
      \Phi(V_n) = (x,n) \big] \\
  &\qquad \leq {3 K \max_{1 \le k \le \delta n} \gamma(k) \over 4}
      \prob(\A(i) \cap \B(i,c_0) \, | \, \Phi(V_n)=(x,n)) \\
  &\qquad \leq {3 A K (\delta n)^{1-\alpha} \over 4}
      \prob(\A(i) \cap \B(i,c_0) \, | \, \Phi(V_n)=(x,n)) \, ,}
where the last inequality is due to our induction hypothesis.

We now proceed to estimate the expectation on the event that
either $\A(i)$ or $\B(i,c_0)$ fail. Recall that we may write
$$ \A(i)^c \cup \B(i,c_0)^c = \bigcup_{a=1}^6 \EE_{\mathrm{(a)}} \cup (\A(i)\cap \B^c(i,c_0)) \, ,$$
where $\EE_{\mathrm{(a)}}$ for $a = 1, \ldots, 6$ were defined in
Section \ref{sec-badtree}. For these estimate we will need the following lemmas.

%

\begin{lemma}
\label{gammacalc}
There exists $C_6 > 0$ such that, assuming the induction hypothesis, for all
$\delta n/4 \le k \leq 2 \delta n$ we have
$$\bar{\gamma}(k; (x,n)) \le (1 + C_6 \alpha) A k^{1-\alpha} \, ,$$
where $\bar{\gamma}$ is defined at \eqref{l:gammabar}.
\end{lemma}

\begin{proof}
By the induction hypothesis
$$ \bar{\gamma}(k; (x,n)) \leq A k^{1-\alpha} G_1(\alpha) \, ,$$
where
$$ G_1(\alpha)
   = \sum_{y \in\Z^d} {\p^{k}(o,y) \p^{n-k}(y,x) \over \p^n(o,x)}
     \Big( {\| y \|^2 \over k} \vee 1 \Big)^\alpha \, .$$
We have that $G_1(0) = 1$, and that
$$ G_1'(\alpha)
   = \sum_{y \in \Z^d: \| y \| > \sqrt{k}}
     {\p^{k}(o,y) \p^{n-k}(y,x) \over \p^n(o,x)}
     \Big( {\| y \|^2 \over k} \Big)^\alpha \log (\| y \|^2/k ) \, .$$
We bound $(\| y \|^2 / k)^\alpha \log (\| y \|^2/k ) \leq C \| y \|^2 / k$
since $\alpha \leq 1/2$ and get that
\eqn{e:G'}
{ G_1'(\alpha)
  \leq C k^{-1} \E \big[ \| S(k) \|^2 \, \big| \, S(n) = x \big] \, .}
Since $k \ge \delta n/4 \ge n_0/4 \ge k_1$ and
$\| x \| \le \sqrt{2n / \delta} = \sqrt{2} n / \sqrt{\delta n} \le 4 n /\sqrt{k}$,
we can apply Proposition \ref{prop:rw-estimates}(i) to the expectation
on the right hand side of \eqref{e:G'}. This gives that
$G_1'(\alpha) \le C$, and the lemma follows.
\end{proof}

For the next lemma, recall the notation of Section \ref{sec-badspatial}. 

\begin{lemma}
\label{gammacalc2}
There exists $C_6 > 0$ such that, assuming the induction hypothesis,
for all $s' \geq s$ we have
$$ \E \Big[ \Reff( (z_{s'-1}, T_{s'-1}) \lra (z_{s'}, T_{s'}) )
      {\bf 1}_{\EE_{(5)}^s} \, \Big| \, \Gtree \big]
   \leq (1 + C_6 \alpha) A t_{s'}^{1-\alpha} \prob(\EE_{(5)}^s| \Gtree) \, .$$
\end{lemma}

\begin{proof}
We appeal to Lemma \ref{spatialbad} and use the induction hypothesis.
When $s' = s$ and $s' \neq 2,\, K+3$ the required quantity is at most
$A t_{s}^{1-\alpha} G_2(\alpha)$ where
$$ G_2(\alpha)
   = \sum_{z} \q_s(z) \sum_{y : \| y \| > \sqrt{t_s}}
     {\p^{t_s}(o,y) \p^{n-T_s+T_0}(y,x-z) \over \p^n(o,x)}
     \Big ( {\|y\|^2 \over t_{s'}} \vee 1 \Big )^{\alpha} \, .$$
By \eqref{e:pes} we have that
$$ G_2(0) = \prob(\EE_{(5)}^s| \Gtree) \, .$$
As in the previous lemma, we have
\eqnspl{e:G'2}
{ G_2'(\alpha)
  &\leq \sum_{z} \q_s(z) \sum_{y : \| y \| > \sqrt{t_s}}
        {\p^{t_s}(o,y) \p^{n-T_s+T_0}(y,x-z) \over \p^n(o,x)}
        {\|y\|^2 \over t_{s}} \\
  &= \frac{1}{t_s} \sum_{z} \q_s(z) \frac{\p^{n - T_{s-1} + T_0}(o,z-x)}{\p^n(o,x)} \\
  &\qquad\quad \times \sum_{y : \| y \| > \sqrt{t_s}} \| y \|^2
        {\p^{t_s}(o,y) \p^{n - T_s + T_0}(y,x-z) \over \p^{n - T_{s-1} + T_0}(o,x-z)} \, .}
Fix $z$ and observe that
\eqnsplst
{ \| x - z \|
  &\leq \| x \| + \| z \|
  \le \sqrt{2n / \delta} + (K+4) \sqrt{\delta n}
  = n \left( \frac{\sqrt{2}}{\sqrt{\delta n}} + \frac{\delta (K+4)}{\sqrt{\delta n}} \right) \\
  &\le n \frac{3}{\sqrt{\delta n}} \, ,}
where in the last step we used \eqref{e:K+3}.
This implies that $\| x - z \| \le 3n / \sqrt{\delta n} \le 3n / \sqrt{t_s}$.
We also have $t_s \ge \delta n/4 \ge n_0/4 \ge k_1$,
where $k_1$ is the constant chosen in Proposition \ref{prop:rw-estimates},
and $t_s \le \delta n \le \delta_1 n$, due to \eqref{e:K+3}.
Hence we can apply Proposition \ref{prop:rw-estimates}(ii) to the sum over
$y$ in \eqref{e:G'2}, and get that
\eqnsplst
{ G_2'(\alpha)
  &\le C \sum_{\substack{z, y :\\ \| y \| > \sqrt{t_s}}}
      \q_s(z) \frac{\p^{t_s}(o,y) \p^{n - T_{s} + T_0}(y,x-z)}{\p^n(o,x)} \\
  &= C \prob(\EE_{(5)}^s| \Gtree) \, .}
This gives the statement of the lemma in the case $s' = s$.

The case $s' > s$ is similar. We appeal to the last statement of
Lemma \ref{spatialbad}, and obtain that the required quantity is
at most $A t_{s'}^{1-\alpha} G_3(\alpha)$ where
\eqnsplst
{ G_3(\alpha)
  &= \sum_{\substack{z, y_s, y : \\ \|y_s\|>\sqrt{t_s}}} \q_s(z)
    {\p^{t_s}(o,y_s) \p^{t_{s'}}(o,y) \p^{n-t_{s'}-T_s}(y, x-z-y_s) \over \p^n(o,x)}
    \Big ( {\|y\|^2 \over t_{s'}} \vee 1 \Big )^{\alpha} \, .}
By \eqref{e:pes} we see (performing the sum over $y$) that
$G_3(0) = \prob(\EE_{(5)}^s| \Gtree)$.
The derivative $G_3'(\alpha)$ can be analyzed similarly
to $G_2'$, this time using Proposition \ref{prop:rw-estimates}(iii).
This yields $G_3'(\alpha) \le C \prob(\EE_{(5)}^s| \Gtree)$,
and proves the statement of the lemma in the case $s' > s$.
\end{proof}

We now proceed with bounding the resistance given $\A(i)^c \cup \B(i,c_0)^c$. Lemmas \ref{badtree1} and \ref{gammacalc}
and the induction hypothesis give that
\eqnsplst
{ &\E \Big[ \Reff(\Phi(X_i) \lra \Phi(X_{i+K})) \, \Big| \,
     \EE_{\mathrm{(1)}}, \Phi(V_n) = (x,n) \Big ] \\
  &\qquad \leq A (K + C_4 \delta) (1 + C_6 \alpha) (\delta n)^{1-\alpha}  \, .}
Lemmas \ref{badtree2} and \ref{gammacalc} give
\begin{equation*}
\begin{split}
  &\E \Big[ \Reff(\Phi(X_i) \lra \Phi(X_{i+K})) \, \Big| \,
      \EE_{\mathrm{(2)}},\, \ell_1,\, \Phi(V_n) = (x,n) \Big] \\
 &\qquad \leq A (1 + C_6 \alpha) \Big[ (\ell_1 - i\delta n)^{1-\alpha}
      + ((i + 1) \delta n - \ell_1)^{1-\alpha} \\
 &\qquad\qquad\qquad\qquad\quad + (K - 1)(\delta n)^{1-\alpha} \Big] + 1 \, .
\end{split}
\end{equation*}
Since ${\ell_1 \over \delta n} - i \in [1/4,3/4]$,
$$ (\ell_1 - i\delta n)^{1-\alpha} + ((i + 1) \delta n - \ell_1)^{1-\alpha}
   \leq (1 + C_7 \alpha) (\delta n)^{1-\alpha} \, ,$$
where $C_7 = (1/2)\log 4 + \sqrt{(3/4)} \log(4/3)$. Hence
\eqnsplst
{ &\E \Big[ \Reff(\Phi(X_i) \lra \Phi(X_{i+K})) \, \Big| \,
      \EE_{\mathrm{(2)}},\, \ell_1,\, \Phi(V_n) = (x,n) \Big] \\
  &\qquad \leq A (1 + (C_6+C_7+C_6C_7)\alpha) K (\delta n)^{1-\alpha} + 1 \, .}
Lemmas \ref{badtree3} and \ref{gammacalc} give
\begin{equation*}
\begin{split}
  &\E \Big[ \Reff(\Phi(X_i) \lra \Phi(X_{i+K})) \, \Big| \,
      \EE_{\mathrm{(3)}},\, \ell_1,\, \Phi(V_n) = (x,n) \Big] \\
  &\qquad \leq A (1 + C_6 \alpha) \Big[ (\ell_1 - i\delta n)^{1-\alpha}
      + ((i + 1) \delta n - \ell_1)^{1-\alpha} \\
  &\qquad\qquad\qquad\qquad\quad + (K - 1 + C_4 \delta)
      (\delta n)^{1-\alpha} \Big] + 1 \\
  &\qquad \leq A (K + C_4\delta + C_7 \alpha ) (1 + C_6\alpha)
      (\delta n)^{1-\alpha} + 1\, .
\end{split}
\end{equation*}
Lemmas \ref{badtree4} and \ref{gammacalc} give
\begin{equation*}
\begin{split}
  &\E \Big[ \Reff(\Phi(X_i) \lra \Phi(X_{i+K})) \, \Big| \,
      \EE_{\mathrm{(4)}},\, \ell_1,\, \ell_2,\, \Phi(V_n) = (x,n) \Big] \\
  &\quad \leq A (1 + C_6 \alpha)) \Big[ (\ell_1 - i \delta n)^{1-\alpha}
      + ((i + 1) \delta n - \ell_1)^{1-\alpha} + (K - 2)(\delta n)^{1-\alpha} \\
  &\quad\qquad\qquad\qquad\quad + (\ell_2 - (i + K - 1) \delta n)^{1-\alpha}
      + ((i + K) \delta n - \ell_2)^{1-\alpha} \Big] + 2 \\
  &\quad \leq A (1 + C_6 \alpha)) (K+2C_7\alpha) (\delta n)^{1-\alpha} + 2  \, .
\end{split}
\end{equation*}

\noindent Lemmas \ref{spatialbad}, \ref{gammacalc2} and the induction
hypothesis give that for any $s = 1, \ldots, K + 4$ and
$s' = 1, \ldots, K + 4$ we have that
\begin{eqnarray*}
 \E \Big[ \Reff((z_{s'-1}, T_{s'-1}) \lra (z_{s'}, T_{s'})) \, \Big| \,
      \EE_{(5)}^s, \Gtree \big] \qquad \qquad   \\
   \qquad \qquad \leq \begin{cases}
      1                   & \text{if $s'=2,K+3$,} \\
      A(t_{s'})^{1-\alpha} & \text{if $s' < s, s'\neq 2, K+3$,}\\
      A (1+C_6 \alpha) (t_{s'})^{1-\alpha} & \text{if $s'\geq s, s'\neq 2, K+3$.}
   \end{cases} \end{eqnarray*}
By the triangle inequality for resistance we get that for all
$s=1, \ldots, K+4$
\eqnsplst
{ &\E \Big[ \Reff(\Phi(X_i) \lra \Phi(X_{i+K})) \, \Big| \,
     \EE_{(5)}^s, \Gtree \Big] \\
  &\qquad \leq A K (1 + C_6\alpha)(1 +2C_7 \alpha) (\delta n)^{1-\alpha} + 2 \, .}
Hence
\begin{equation*}
\begin{split}
  &\E \Big[ \Reff(\Phi(X_i) \lra \Phi(X_{i+K})) \, \Big| \,
     \EE_{\mathrm{(5)}},\, \ell_1,\, \ell_2,\, \Phi(V_n) = (x,n) \Big] \\
  &\qquad \leq A K (1 + C_6\alpha)(1 +2C_7 \alpha) (\delta n)^{1-\alpha} + 2 \, .
\end{split}
\end{equation*}
By Lemma \ref{badtree6} and the induction hypothesis (recall
$\EE_{\mathrm{(7)}} = \A(i) \cap \B(i,c_0)^c$):
\eqnsplst
{ \E \Big [ \Reff(\Phi(X_i) \lra \Phi(X_{i+K}))  \, \Big| \,
     \EE_{\mathrm{(6)}} \cup \EE_{\mathrm{(7)}}, \Gtree \Big]
  \leq A (K + 2C_7\alpha) (\delta n)^{1-\alpha} +2 \, .}
%
Putting these together gives that there exists $C_8=C_8(K)>0$ such that
\eqnsplst
{ &\E \Big[ \Reff(\Phi(X_i) \lra \Phi(X_{i+K}))  \, \Big| \,
      \A(i)^c \cup \B(i,c_0)^c,\, \Phi(V_n) = (x,n) \Big] \\
  &\qquad \leq A (K + C_8(\delta +\alpha)) (\delta n)^{1-\alpha} + 2 \, .}
This together with \eqref{goodbound} yields
\eqnsplst
{ &\E \Big[ \Reff(\Phi(X_i) \lra \Phi(X_{i+K}))  \, \Big| \,
      \Phi(V_n) = (x,n) \Big] \\
  &\qquad \leq {3 A K (\delta n)^{1-\alpha} \over 4}
      \prob(\A(i)\cap \B(i,c_0) \, | \, \Phi(V_n) = (x,n) ) \\
  &\qquad\quad + \Big [ A (K + C_8(\delta+ \alpha)) (\delta n)^{1-\alpha} + 2 \Big ]
      \prob(\A(i)^c \cup \B(i,c_0)^c \, |\, \Phi(V_n) = (x,n)) \, .}
By Theorem \ref{goodthingshappen} there exists a constant
$c = c(K) \in (0,1)$ such that the last quantity is at most
$$ A (\delta n)^{1-\alpha} \Big[  {c3K \over 4} +
     (1 - c) (K + C_8(\delta + \alpha)) \Big] \, .$$
We now choose $\delta_0$ and $\alpha$ (depending only on $K=K_0$).
In addition to the already required \eqref{e:K+3}, let $\delta_0$
satisfy:
\be
\label{magic1}
  \delta_0 \le {c \over 16 C_8} \, , \qquad\qquad
  2 K \delta_0 (1 + C_6) + 2 \delta_0 (1 + C_6) < \frac{c}{16} \, ,
\ee
Let $\alpha > 0$ satisfy:
\be
\label{magic2}
  \alpha \le \delta_0 \, , \qquad\qquad
  \Big ( 1 - {c \over 16} \Big ) \delta_0^{-\alpha} \le 1 \, .
\ee
The first condition on $\delta_0$ in \eqref{magic1} gives that
\eqnspl{e:bulk-bnd}
{ &\E \Big[ \Reff(\Phi(X_i) \lra \Phi(X_{i+K}))  \, \Big| \,
      \Phi(V_n) = (x,n) \Big] \\
  &\qquad \leq A K (1 - c/8) (\delta n)^{1-\alpha} \\
  &\qquad \le A n^{1-\alpha} \delta^{-\alpha} K \delta (1 - c/8) \, ,}
for $i = 0, K, \dots, (N-1)K$. For the final stretch,
Lemmas \ref{lem:final-stretch} and \ref{gammacalc}
and the induction hypothesis gives:
\eqnspl{e:final-stretch-bnd}
{ &\E \Big[ \Reff(\Phi(X_{i^{\mathrm{last}}}) \lra \Phi(V_n)) \, \Big| \,
      \Phi(V_n) = (x,n) \Big] \\
  &\qquad \le K' A (\delta n)^{1-\alpha} (1 + C_6 \alpha)
      + A (\tilde{n})^{1-\alpha} ( 1 + C_6 \alpha ) \, .}
Using $K' \le 2K$ and $\tilde{n} \le 2 \delta n$ and the
second requirment on $\delta_0$ in \eqref{magic1},
the right hand side of \eqref{e:final-stretch-bnd}
is at most
\eqnsplst
{ &A n^{1-\alpha} \delta^{-\alpha}
      \big( 2 K \delta (1 + C_6) + (2 \delta)(1 + C_6) \Big) \\
  &\qquad \le A n^{1-\alpha} \delta^{-\alpha} (c/16) \, .}
We sum \eqref{e:bulk-bnd} over all sequences using the triangle inequality
and add \eqref{e:final-stretch-bnd}. This gives
\eqnspl{e:gamma-nx-bnd}
{ \gamma(n,x)
  &= \sup_{m \geq 2n} \gamma_m(n,x) \\
  &\leq A n^{1-\alpha} \delta^{-\alpha} \Big[ N K \delta (1 - c/8) + c/16 \Big] \\
  &\le A n^{1-\alpha} \delta^{-\alpha} (1 - c/16) \, ,}
where we used $N K \delta \le 1$.

To conclude, note that if in the definition \eqref{e:def-delta}
of $\delta$ we have $\delta_0 \le n / \| x \|^2$, then
$\delta^{-\alpha} \le \delta_0^{-\alpha}$, and due to the choice
of $\alpha$, \eqref{e:gamma-nx-bnd} yields $\gamma(n,x) \leq A n^{1-\alpha}$.
If we have $n / \| x \|^2 < \delta_0$, then $\delta^{-\alpha} \le (\| x \|^2/n)^\alpha$,
and \eqref{e:gamma-nx-bnd} yields $\gamma(n,x) \leq A n^{1-\alpha} (\| x \|^2 / n)^\alpha$
This concludes the induction and our proof.
\qed

\section*{Acknowledgements} We thank Ori Gurel-Gurevich, Gady Kozma and Gordon Slade for useful conversations. This research was supported by NSF and NSERC grants.

\end{document}